\documentclass[12pt, a4paper, oneside]{article}
\usepackage{amsthm, amssymb, bm, color, framed, graphicx, hyperref, mathrsfs,bbm}
\usepackage[left=2.0cm,right=2.0cm,top=2.5cm,bottom=3.5cm]{geometry}
\usepackage{amsmath}
\usepackage{animate}
\usepackage{graphicx,comment,float,CJKutf8}
\numberwithin{equation}{section}

\title{\textbf{Complexity of flows with interactions and stochastic analysis}}
\author{Yifan Shi \\
HSE University\\
yshi@edu.hse.ru
}
\date{}
\definecolor{shadecolor}{RGB}{255, 255, 255}

\usepackage[TS1,OT1,T1]{fontenc}
\newtheorem{theorem}{Theorem}[section]
\newtheorem{lemma}[theorem]{Lemma}
\newtheorem{corollary}[theorem]{Corollary}
\newtheorem{proposition}[theorem]{Proposition}

\theoremstyle{definition}
\newtheorem{definition}[theorem]{Definition}

\newtheorem{assumption}[theorem]{Assumption}

\newtheorem{remark}[theorem]{Remark}

\begin{document}
\maketitle
\tableofcontents
\newpage
\begin{abstract}

We study stochastic flows with interaction through an infinite-dimensional
semigroup approach. By means of a kernel lifting procedure, we represent
the solution to measure-valued stochastic equation as a Hilbert-valued Markov process.
We identify its generator, construct the associated Markov semigroup, and
establish a dense range condition using measure polynomial functionals.
The results provide a framework for the analysis of stability of stochastic equations
with interaction.

\end{abstract}

\section{Introduction}

Equations with interaction provide a framework for describing systems of particles whose motion depends not only on the position of an individual particle but also on the spatial distribution of the whole system. A basic form of this construction can already be seen for a finite system of interacting particles. Let $x_1(t),\ldots,x_N(t)$ denote the positions of $N$ particles and consider an interaction term of the form
\begin{equation}
\frac{1}{N}\sum_{j=1}^N h(x_i(t),x_j(t)).
\end{equation}
Introducing the empirical measure
\begin{equation}
\mu_t^N=\frac{1}{N}\sum_{j=1}^N\delta_{x_j(t)},
\end{equation}
the interaction term can be written as
\begin{equation}
\frac{1}{N}\sum_{j=1}^N h(x_i(t),x_j(t))
=
\int_{\mathbb R^d}h(x_i(t),v)\,\mu_t^N(\mathrm{d}v).
\end{equation}
Thus, the interaction between finitely many particles can be expressed in terms of the empirical distribution of the whole system. 

This point of view is closely related to the theory of stochastic flows with interaction and measure-valued processes developed by A. A. Dorogovtsev. In particular, A. A. Dorogovtsev~\cite{Dorogovtsev_2003} considered measure-valued processes represented in the form
\begin{equation}
\mu_t=\mu_0\circ x(\cdot,t)^{-1},
\end{equation}
where the random map $x(\cdot,t)$ describes the motion of the particles carrying the mass.

A general stochastic differential equation with interaction considered in this paper is
\begin{definition}
The equation with interaction is
\begin{equation}
\begin{cases}
\mathrm{d}x(u,t)
=
a(x(u,t),\mu_t,t)\mathrm{d}t
+
\displaystyle\int_{\mathbb R}
b(x(u,t),\mu_t,t,q)W(\mathrm{d}t,\mathrm{d}q),
\\
x(u,0)=u,\qquad u\in\mathbb R^d,
\\
\mu_t=\mu_0\circ x(\cdot,t)^{-1},\qquad t\geq0.
\end{cases}
\label{eq:interacting}
\end{equation}
Here $\mu_t$ is the image of the initial measure $\mu_0$ under the random flow $x(\cdot,t)$, and $W$ is a Wiener sheet.
\end{definition}

%The essential difference from an ordinary stochastic differential equation is that the coefficients depend on the evolving measure $\mu_t$. This self-consistent structure is also the reason why equations with interaction should be distinguished from classical McKean--Vlasov equations. In the latter setting, the measure is usually interpreted as the law of a random variable, whereas in the present framework $\mu_t$ is the image of the initial mass distribution under the same stochastic flow that determines the trajectories of the particles. 

The existence and uniqueness theory for such equations was developed in a number of works, beginning with the constructions of P. Kotelenez and A. A. Dorogovtsev and subsequent developments. See, in particular, A. A. Dorogovtsev~\cite{Dorogovtsev_2003} and the monograph~\cite{Dorogovtsev_2024}. 

To formulate the existence theory for the solution of equation~(\ref{eq:interacting}), one naturally works on spaces of probability measures equipped with Wasserstein metrics. Let $\mathfrak M$ denote the space of probability measures on the Borel $\sigma$-field $\mathcal B(\mathbb R^d)$. For $\mu,\nu\in\mathfrak M$, let $C(\mu,\nu)$ be the set of probability measures on $\mathbb R^d\times\mathbb R^d$ having $\mu$ and $\nu$ as its marginal distributions.

\begin{definition}\cite{Dudley2002}
For $n\geq1$, define
\begin{equation}
\mathfrak M_n
=
\left\{
\mu\in\mathfrak M:
\int_{\mathbb R^d}\rho(u,v)^n\,\mu(\mathrm{d}v)<+\infty
\text{ for all }u\in\mathbb R^d
\right\}.
\end{equation}
The Wasserstein distance of order $n$ between $\mu$ and $\nu$ is
\begin{equation}
\gamma_n(\mu,\nu)
=
\left(
\inf_{\varkappa\in C(\mu,\nu)}
\iint_{\mathbb R^d\times\mathbb R^d}
\rho(u,v)^n
\,\varkappa(\mathrm{d}u,\mathrm{d}v)
\right)^{1/n}.
\end{equation}
\end{definition}

In the setting considered here, one may use the following formulation from~\cite{Dorogovtsev_2024}.

\begin{theorem}\label{existence}
The solution of equation~\eqref{eq:interacting} exists and is unique if there exists $C>0$ such that, for all $u_1,u_2\in\mathbb R^d$, $\nu_1,\nu_2\in\mathfrak M_m$ and $t\geq0$,
\begin{equation}
\begin{split}
&\|a(u_1,\nu_1,t)-a(u_2,\nu_2,t)\|+
\left(
\int_{\mathbb R^d}
\|b(u_1,\nu_1,t,q)-b(u_2,\nu_2,t,q)\|^2
\,\mathrm{d}q
\right)^{1/2}\\
\leq
&C\left(
\|u_1-u_2\|
+
\gamma_m(\nu_1,\nu_2)
\right),
\end{split}
\end{equation}
and $a$ and $b$ are continuous with respect to all their variables.
\end{theorem}

The theory developed in~\cite{Dorogovtsev_2003,Dorogovtsev_2024} provides a natural starting point for the present work. However, our objective is different. Rather than studying only the measure-valued process $(\mu_t)_{t\geq0}$, we seek a Hilbert-space representation of its dynamics and use this representation to study the operator-theoretic structure of the resulting stochastic process.

We therefore introduce a kernel-based lifting of the measure-valued process since the space of probability measures is not a linear vector space.  

Let $h:\mathbb R^d\to\mathcal H$ be a smooth map into a real separable Hilbert space such that
\begin{equation}
(h(u),h(v))_{\mathcal H}
=
\Gamma(u-v),
\label{eq:kernel_inner}
\end{equation}
where $\Gamma$ is a positive-definite kernel. We define the lifting map by
\begin{definition}
\begin{equation}
H(\mu)
=
\int_{\mathbb R^d}h(u)\,\mu(\mathrm{d}u).
\label{eq:lift}
\end{equation}
\end{definition}

As it will be shown later in Lemma~\ref{separates}, under suitable assumptions on the kernel, the map $H$ separates probability measures. In particular, if the Fourier transform $\widehat{\Gamma}$ is nonzero everywhere, then $H$ is injective. Thus, on the class of measures under consideration, the measure-valued process can be represented by the Hilbert-valued process
\begin{equation}
h_t=H(\mu_t).
\end{equation}
This representation does not replace the Wasserstein description of the original process; rather, it provides a complementary linear realization in which the measure-valued dynamics can be studied using Hilbert-space stochastic analysis.

This lifting leads to the main object of the present paper. We study the Markov process $(h_t)_{t\geq0}$ and its infinitesimal generator. Under the assumptions of Theorem~\ref{existence}, we obtain the following representation.

\begin{theorem}
Assume that the coefficients satisfy the conditions of Theorem~\ref{existence}. Then $\{h_t\}_{t\geq0}$ is a Markov process, and its infinitesimal generator is given by
\begin{equation}
(\mathcal L_s\Phi)(h)
=
D\Phi(h)
\bigl[
A_s(\mu(h))+G_s(\mu(h))
\bigr]
+
\frac12
\sum_{k=1}^{\infty}
D^2\Phi(h)
\bigl(
\Sigma_{s,k}(\mu(h)),
\Sigma_{s,k}(\mu(h))
\bigr),
\end{equation}
where explicit formulas for $A_s$, $G_s$, and $\Sigma_{s,k}$ are given in Section~\ref{generator}. The associated transition semigroup is strongly continuous and contractive.
\end{theorem}

A central issue is then to identify a sufficiently large class of test functions for the generator. For this purpose we use polynomial functionals of measures. Let $\mathcal P(K)$ denote the algebra generated by the linear functionals
\begin{equation}
\mu\mapsto\langle f,\mu\rangle
=
\int_{\mathbb R^d}f(u)\,\mu(\mathrm{d}u),
\qquad f\in C_b(\mathbb R^d).
\end{equation}
These functionals are naturally adapted to the measure-valued structure of the equation: applying the generator to a product of such functionals produces again a finite combination of measure polynomial functionals. After transferring them through the injective map $H$, they provide a concrete class of functions on the lifted state space.

We prove that an appropriate weighted completion of this polynomial class satisfies a dense range condition for the generator. More precisely, there exists a weighted polynomial space $\mathcal B_\rho$ continuously embedded in $C(G)$ such that, for every $\lambda_0>M_\rho$,
\begin{equation}
\overline{\operatorname{Ran}(\lambda_0I-\mathcal L)}^{C(G)}
=
C(G).
\label{eq:dense_range}
\end{equation}
Consequently, the polynomial functionals form a core for the generator.

The measure-valued and stochastic-flow aspects of equations with interaction have also been investigated from several other directions. A. A. Dorogovtsev and Wei{\ss}~\cite{dorogovtsev_weiss_2023} studied intermittency phenomena for mass distributions of stochastic flows with interaction. Their analysis uses the evolution of mass densities and Lyapunov exponents and establishes intermittency under dissipativity assumptions. This is closely related to the asymptotic part of the present work, but the objects of study are different. Whereas~\cite{dorogovtsev_weiss_2023} investigates the spatial concentration and intermittency of the mass distribution itself, we study the asymptotic behavior after the measure-valued process has been represented in a Hilbert space.

In another direction, recent work has considered ergodicity and functional expansions for the measure-valued processes generated by equations with interaction. For example, A. A. Dorogovtsev and Wei{\ss}~\cite{dorogovtsev_weiss_2026} study equations with interaction on compact Riemannian manifolds and derive an Itô--Wiener type decomposition for the associated measure-valued process. Their approach remains directly on the Wasserstein space and exploits differentiability with respect to the measure variable. Our approach is complementary: instead of differentiating functionals on the Wasserstein space, we use an injective Hilbert-space realization of the measure and study the resulting linear infinite-dimensional state space.

We next consider asymptotic stability. A useful notion in the theory of equations with interaction is shift compactness, which naturally appears when the system possesses translation invariance. Following~\cite{Dorogovtsev_2024}, we use this notion to describe compactness of families of random measures modulo spatial translations.

\begin{definition}
A set of random measures $\{\mu_\alpha;\alpha\in\mathcal A\}\subset\mathfrak M_n$ is called shift compact if, for every $\alpha\in\mathcal A$, there is a random vector $u_\alpha\in\mathbb R^d$ such that
$\{\mu_\alpha-u_\alpha:\alpha\in\mathcal A\}$ is weakly compact in $\mathfrak M_n$.
\end{definition}

For dissipative interaction kernels and sufficiently small noise, A. A. Dorogovtsev~\cite{Dorogovtsev_2024} obtained shift-compactness of the family $\{\mu_t:t\geq0\}$ under conditions of the form
\begin{equation}
(u-v,\varphi(u)-\varphi(v))
\leq
-\alpha\|u-v\|^2
\end{equation}
and
\begin{equation}
2\alpha-B^2(2n+1)\geq0.
\end{equation}
In the present paper we further investigate the relation between contraction of the underlying measure-valued dynamics and asymptotic properties of observables in the Hilbert representation.

Under the above assumptions, we obtain convergence results for the centered measures and for the corresponding Gaussian observables $\eta_k(t)=(H(\mu_t),\xi_k)$. In particular, we prove convergence of the covariance structure in the form
\begin{equation}
\operatorname{Cov}(\eta_k(t),\eta_k(s))
\longrightarrow
\operatorname{Cov}_{\infty}(t-s),
\qquad k\to\infty.
\label{eq:cov}
\end{equation}
We also construct a feedback drift for which a prescribed functional $F(\mu_t)$ converges to zero in $L^2$, yielding a corresponding basin-of-attraction result.

Finally, we apply the lifting construction to point vortex systems. Such systems provide a particularly natural deterministic example of an equation with interaction: the velocity of each vortex depends on the positions of all other vortices. For infinitesimal vortices, the Helmholtz system is
\begin{equation}
\mathrm{d}y_k(t)
=
\sum_{\substack{j=1\\j\neq k}}^N
p_j
\nabla^\perp
\ln\|y_k(t)-y_j(t)\|
\,\mathrm{d}t,
\qquad
1\leq k\leq N.
\label{vortex dynamics}
\end{equation}
where $p_j$ is the infinitesimal vorticity. Introducing the empirical vorticity measure
\begin{equation}
\mu_t^N
=
\sum_{i=1}^N\kappa_i\delta_{x_i(t)},
\qquad
\sum_{i=1}^N\kappa_i=1,
\label{eq:empirical}
\end{equation}
allows we use the Hilbert lifting to study the effective dimension of clustered vortex configurations.

Suppose that the vortices form $m$ clusters with intra-cluster diameter at most $\delta$ and inter-cluster distance at least $R$. For the Gaussian kernel, we prove that the quadratic entropy of the lifted configuration satisfies
\begin{equation}
\mathcal H_2(A_t^N)
\leq
m
+
CN\delta^2
+
Cm^2e^{-R^2/2}.
\end{equation}
Thus, when the clusters are well separated and internally concentrated, the effective Hilbert-space complexity is controlled by the number of clusters rather than directly by the number of particles.

The paper is organized as follows. In Section~\ref{sec:lifting}, we introduce a kernel-based Hilbert-space representation of the measure-valued process and establish conditions under which the lifting map is injective and continuous. We also study the properties of the lifted state space and prove a convergence result for the lifted processes. In Section~\ref{sec:markov}, we prove that the lifted process is Markovian, derive its infinitesimal generator, and construct the associated transition semigroup. We further establish a dense range condition for a weighted space of polynomial functionals and show that these functionals form a core for the generator. In Section~\ref{sec:stability}, we investigate the asymptotic behavior of the lifted dynamics under dissipativity assumptions and construct a feedback drift that yields a prescribed basin of attraction. Finally, in Section~\ref{sec:cluster}, we apply the kernel lifting to Helmholtz point vortex systems and obtain estimates for the quadratic entropy of the corresponding lifted dynamics.

\section{Kernel Lift to Equation with Interaction}\label{sec:lifting}
\subsection{Lifting from space of measures to Hilbert space}
Let $\mathcal{H}$ be a real, separable, infinite-dimensional Hilbert space. 
Assume that $h:\mathbb{R}^d \to \mathcal{H}$ is twice continuously differentiable, i.e., $h \in C^2(\mathbb{R}^d;\mathcal{H})$, and satisfies
\begin{equation}
\langle h(u), h(v) \rangle_{\mathcal{H}} = \Gamma(u-v),
\end{equation}
where $\Gamma:\mathbb{R}^d \to \mathbb{R}$ is a continuous positive-definite kernel.

Define the mapping $\mathrm{H}:\mathfrak{M}\to\mathcal{H}$ by
\begin{equation}
\mathrm{H}(\mu)=\int_{\mathbb{R}^d} h(u)\,\mu(\mathrm{d}u).
\end{equation}
We refer to the mapping $\mathrm{H}$ as a \emph{lifting} from $\mathfrak{M}$ into the Hilbert space $\mathcal{H}$.

\subsubsection{One to one correspondence with measures}
\begin{lemma}\label{separates}
If the Fourier transform of the kernel function $\Gamma$ satisfies that for arbitrary $\lambda\in\mathbb{R}^d$
\begin{equation}
\hat\Gamma(\lambda)\not=0,
\end{equation} 
then $\mathrm{H}$ is an injection.
\end{lemma}

\begin{proof}
Assume there exist $\mu^1,\mu^2\in \mathfrak{M}$ such that $\mathrm{H}(\mu^1)=\mathrm{H}(\mu^2)$. Then
\begin{equation}
\Gamma*(\mu^1-\mu^2)(v)=\int_{\mathbb{R}^d}\Gamma(v-u)(\mu^1-\mu^2)(\mathrm{d}u)=\left(h(v),\int_{\mathbb{R}^d}h(u)\mu^1(\mathrm{d}u)-\int_{\mathbb{R}^d}h(u)\mu^2(\mathrm{d}u)\right)=0
\end{equation}
Take the Fourier transform from both sides, then
\begin{equation}
\hat\Gamma\cdot\widehat{\mu^1-\mu^2}\equiv0.
\end{equation}
Due to $\hat\Gamma(\lambda)\not=0$ for any $\lambda$, one can obtain
\begin{equation}
\mu^1=\mu^2
\end{equation}
which completes the proof.
\end{proof}

\subsubsection{Continuity of the lifting}

\begin{lemma}
Assume that $\Gamma$ is bounded and continuous. Then the mapping 
$\mathrm H:(\,\mathfrak M,\gamma_0) \to (\,\mathcal H,\|\cdot\|_{\mathcal{H}})$ is continuous.
\end{lemma}

\begin{proof}
Since the topology generated by $\gamma_0$ is equivalent to weak topology \cite{Villani2003Topics}, it is enough to check that $\mathrm{H}$ is continuous in $\mathfrak{M}$ equipped with weak topology.

Let $\mu_n \Rightarrow \mu$ weakly as $n\to\infty$. For any $\varphi \in \mathcal H$, we have
\begin{equation}
\langle \varphi, \mathrm H(\mu_n) \rangle_{\mathcal H}
= \int_{\mathbb{R}^d} \langle \varphi, h(u) \rangle_{\mathcal H} \,\mu_n(\mathrm d u).
\end{equation}
Since $h$ is continuous and bounded, the function 
$u \mapsto \langle \varphi, h(u) \rangle_{\mathcal H}$ is bounded and continuous. 
Hence,
\begin{equation}
\langle \varphi, \mathrm H(\mu_n) \rangle_{\mathcal H}
\to
\langle \varphi, \mathrm H(\mu) \rangle_{\mathcal H}, \quad \text{as } n\to\infty,
\end{equation}
which implies that $\mathrm H(\mu_n)\rightharpoonup \mathrm H(\mu)$ weakly in $\mathcal H$.

It remains to show convergence of norms. Observe that
\begin{equation}
\|\mathrm H(\mu_n)\|_{\mathcal H}^2
= \iint_{\mathbb{R}^d \times \mathbb{R}^d} \Gamma(u-v)\,\mu_n(\mathrm d u)\mu_n(\mathrm d v).
\end{equation}
Since $\mu_n \Rightarrow \mu$, it follows that $\mu_n \otimes \mu_n \Rightarrow \mu \otimes \mu$ weakly on $\mathbb{R}^d \times \mathbb{R}^d$. 
Because $\Gamma(u-v)$ is bounded and continuous on $\mathbb{R}^{2d}$, we obtain
\begin{equation}
\iint \Gamma(u-v)\,\mu_n(\mathrm d u)\mu_n(\mathrm d v)
\to
\iint \Gamma(u-v)\,\mu(\mathrm d u)\mu(\mathrm d v),\quad \text{as } n\to\infty.
\end{equation}
Thus,
\begin{equation}
\|\mathrm H(\mu_n)\|_{\mathcal H} \to \|\mathrm H(\mu)\|_{\mathcal H}.
\end{equation}
which completes the proof.
\end{proof}

\subsection{Properties of the lifted set}
For $K\subset\mathfrak{M}$ define the \emph{lifted set}
\begin{equation}\label{eq:G_def}
\mathcal G
=
\left\{
\,\mathrm{H}(\mu)
:
\mu\in K
\right\}
\subset \mathcal H .
\end{equation}

\begin{proposition}
The lift set $\mathcal G$ is bounded. Moreover, if $K$ is a compact set under weak convergence, then $\mathcal G$ is a compact subset of $\mathcal{H}$.
\end{proposition}

\begin{proof}
Due to $\mathcal G\subset B_\mathcal{H}\left(0,\sqrt{\Gamma(0)}\right)$, $\mathcal{G}$ is bounded. If $K$ is a compact set under weak convergence, then $\mathcal G$ is a compact subset of $\mathcal{H}$ since $\mathrm{H}$ is a continuous mapping from $\mathfrak{M}$ to $\mathcal{H}$.
\end{proof}

\subsection{Convergence}
Now we state the continuous dependence of $H(\mu_t)$ for $\mu_0$.
\begin{theorem}\label{convergence thm}
Assume the conditions of Theorem~\ref{existence} hold, and let
\(h:\mathbb R^d\to\mathcal H\) be bounded and globally Lipschitz.
If 
\begin{equation}
\mathbb{E}\,\gamma_2(\mu_0^n,\mu_0)\to0,\quad n\to \infty
\end{equation}
where $\gamma_2$ is 2-Wasserstein distance, then for every \(t>0\),
\begin{equation}
\mathrm H(\mu_t^n)\to \mathrm H(\mu_t)
\quad\text{in }L^2(\Omega;\mathcal H),\quad n\to\infty
\end{equation}
where
\begin{equation}
\mu_t^n=\mu_0^n\circ(x^n(\cdot,t))^{-1},
\qquad
\mu_t=\mu_0\circ(x(\cdot,t))^{-1}.
\end{equation}
\end{theorem}
\begin{remark}
Condition of Theorem \ref{convergence thm} holds if, for example, the kernel $\Gamma$ satisfies $\Gamma\in C^2(\mathbb{R}^d)$ and 
\begin{equation}
\sup_{z\in\mathbb R^d}
\left|
\frac{\partial^2\Gamma}{\partial z_i\partial z_j}(z)
\right|
<\infty,
\qquad 1\le i,j\le d.
\end{equation}
\end{remark}
\begin{proof}
By Lemma 2.4.1 in \cite{Dorogovtsev_2024}, for every fixed \(t>0\) there exists a constant \(C_t>0\) such that
\begin{equation}
\mathbb{E}\,\gamma_2(\mu_t^n,\mu_t)
\le
C_t\mathbb{E}\,\gamma_2(\mu_0^n,\mu_0).
\end{equation}

Since \(h:\mathbb R^d\to\mathcal H\) is globally Lipschitz, there exists \(L>0\) such that
\begin{equation}
\|h(x)-h(y)\|_{\mathcal H}
\le
L\|x-y\|,
\qquad x,y\in\mathbb R^d.
\end{equation}

By the definition of \(\mathrm H\), we have
\begin{equation}
\mathrm H(\mu_t^n)-\mathrm H(\mu_t)
=
\int_{\mathbb R^d} h(x)\,\mu_t^n(\mathrm dx)
-
\int_{\mathbb R^d} h(x)\,\mu_t(\mathrm dx).
\end{equation}
then for fixed $\varkappa\in C(\mu^n_t,\mu_t)$
\begin{equation}
\begin{split}
&\|\mathrm H(\mu_t^n)-\mathrm H(\mu_t)\|_{\mathcal H}\\
=&\left\|\int_{\mathbb R^d} h(x)\,\mu_t^n(\mathrm dx)
-
\int_{\mathbb R^d} h(x)\,\mu_t(\mathrm dx)\right\|_{\mathcal H}\\
=&\left\|\int_{\mathbb R^d\times\mathbb{R}^d} (h(x)-h(y))\,\varkappa(\mathrm dx,dy)\right\|_\mathcal H\\
\le&L\int_{\mathbb R^d\times\mathbb{R}^d} \left\|x-y\right\|_\mathcal H\,\varkappa(\mathrm dx,dy).\\
\end{split}
\end{equation}
Since $\varkappa$ is arbitrary, then
\begin{equation}
\|\mathrm H(\mu_t^n)-\mathrm H(\mu_t)\|_{\mathcal H}\le L\,\gamma_1(\mu_t^n,\mu_t)
\end{equation}

Moreover,
\begin{equation}
\gamma_1(\mu_t^n,\mu_t)\le \gamma_2(\mu_t^n,\mu_t).
\end{equation}

Hence,
\begin{equation}\label{ine for H^2}
\mathbb{E}\,\|\mathrm H(\mu_t^n)-\mathrm H(\mu_t)\|_{\mathcal H}
\le
L C_t\mathbb{E}\,\gamma_2(\mu_0^n,\mu_0).
\end{equation}

Since \(\mathbb{E}\,\gamma_2(\mu_0^n,\mu_0)\to0\), combining estimation (\ref{ine for H^2}), we deduce
\begin{equation}
\mathbb E\|\mathrm H(\mu_t^n)-\mathrm H(\mu_t)\|_{\mathcal H}^2
\le 2\sup_{x\in\mathbb{R}^d}\|h(x)\|\cdot\mathbb E\|\mathrm H(\mu_t^n)-\mathrm H(\mu_t)\|_{\mathcal H}
\to 0,\text{\quad as }n\to+\infty
\end{equation}
namely
\begin{equation}
\mathrm H(\mu_t^n)\to \mathrm H(\mu_t)
\quad \text{in } L^2(\Omega;\mathcal H),\text{\quad as }n\to+\infty
\end{equation}

This completes the proof.
\end{proof}

\section{$h_t$ as a Markov process}\label{sec:markov}
In and after this section, we consider the Hilbert space $\mathcal{H}=L_2(\mathbb{R}^d)$ and denote $G:=\mathrm{H}(\mathfrak{M}_m)$. We also consider the Gaussian kernel $\Gamma(\cdot)=\exp\{-\|\cdot\|^2\}$. For simplicity, we define
\begin{equation}
    h_t:=\mathrm{H}(\mu_t)=\int_{\mathbb{R}^d}h(u)\mu_t(\mathrm{d}u).
\end{equation}

\subsection{Markov property}

\begin{theorem}\label{Markovian}
Assume the conditions of Theorem~\ref{existence} hold. Then $\{h_t\}$ is a Markov process in $G$ relative to the flow $\{\mathcal F_t\}$ generated by the Wiener sheet.
\end{theorem}

\begin{proof}
Define the stochastic kernel $K_t(\cdot,\cdot):G\times\mathcal{B(H)}\to \mathbb{R}$ by
\begin{equation}
K_t(h_0,\Delta)=\mathbb{P}\left\{h_t\in\Delta\right\}=\mathbb{P}\left\{\int_{\mathbb{R}^d}h(u)\mathrm{d}\mu_{t}\in\Delta\right\},\Delta\in\mathcal{B}(\mathcal{H}).
\end{equation}
where $\{\mu_t;t\ge0\}$ is the measure-valued process obtained from the solution of equation (\ref{eq:interacting}) with the initial condition $\mu_0$.

Let us prove $K_t(\cdot,\cdot)$ is the stochastic kernel on $\mathcal{H}$. First $K_t(h_0,\cdot)$ is the distribution of $h_t$, so it's a probability measure. Then we will check the measurability of $K_t$ with respect to $h_0$.

For every bounded continuous function $f:\,\mathcal{H}\to\mathbb{R}$, the integral
\begin{equation}
\int_{\mathcal{H}}f(\nu)K_t(h_0,\mathrm{d}\nu)
\end{equation}
is the continuous function of $h_0$ since for any sequence \(h_0^n\to h_0\) in \(\mathcal H\), the corresponding solutions satisfy
\begin{equation}
h_t^n \to h_t
\qquad \text{in } L^2(\Omega;\mathcal H),
\end{equation}
by Theorem~\ref{convergence thm}. Hence,
\begin{equation}
f(h_t^n)\to f(h_t)
\qquad \text{in probability},
\end{equation}
for every bounded continuous function \(f:\mathcal H\to\mathbb R\). Since \(f\) is bounded, the dominated convergence theorem yields
\begin{equation}
\mathbb E f(h_t^n)\to \mathbb E f(h_t).
\end{equation}
Therefore,
\begin{equation}
\int_{\mathcal H} f(\nu)\,K_t(h_0^n,d\nu)
\to
\int_{\mathcal H} f(\nu)\,K_t(h_0,d\nu),
\end{equation}
which proves that
\begin{equation}
h_0\mapsto \int_{\mathcal H} f(\nu)\,K_t(h_0,d\nu)
\end{equation}
is a continuous mapping from $\mathcal{H}$ to $\mathbb{R}$.

Next, we verify the Markov property. Let $s,t\ge0$ and $\Delta\in\mathcal{B}(\mathcal{H})$ be arbitrary. We aim to show
\begin{equation}
\mathbb{P}\{h_{t+s}\in\Delta \mid \mathcal{F}_t\}
=
K_s(h_t,\Delta).
\end{equation}
Since $\{\mu_t\}_{t\ge0}$ is a homogeneous Markov process in $\mathfrak{M}_m$ relative to the $\mathcal{F}_t$ \cite{Dorogovtsev_2024} and $\mathrm H$ is the injection,
\begin{equation}
\mathbb P\{h_{t+s}\in\Delta\mid\mathcal F_t\}
=
\mathbb P\{\mu_{t+s}\in H^{-1}(\Delta)\mid\mathcal F_t\} 
=
P_s(\mu_t,\mathrm{H}^{-1}(\Delta))=K_s(h_t,\Delta)
\end{equation}
where \(P_s\) denotes the transition kernel of \(\{\mu_t\}_{t\ge0}\).

Thus \(\{h_t\}_{t\ge0}\) is a Markov process.

\end{proof}
\subsection{Generator}\label{generator}
Let $\{e_k\}_{k\ge1}$ be an orthonormal basis of $L^2(\mathbb R^d)$ and set
\begin{equation}\label{eq:beta_k_def}
\beta_t^k:=\int_0^t\int_{\mathbb R^d} e_k(q)\,W(ds,dq),\qquad k\ge1 .
\end{equation}
Since $\langle e_j,e_k\rangle_{L^2}=\delta_{jk}$, $(\beta^k)_{k\ge1}$ are
independent standard Brownian motions. Let $V$ be a separable Hilbert space
($V=\mathbb R^d$ or $V=\mathcal H$ below). For $f\in L^2(\mathbb R^d;V)$
define $\langle f,e_k\rangle:=\int_{\mathbb R^d} f(q)\,e_k(q)\,dq\in V$ and
\begin{equation}\label{eq:noise_expansion}
\int_0^t\!\!\int_{\mathbb R^d} f(q)\,W(ds,dq)
:=\sum_{k=1}^{\infty}\int_0^t \langle f,e_k\rangle\,d\beta_s^k ;
\end{equation}
Parseval's identity gives
\begin{equation}\label{eq:Parseval}
\sum_{k\ge1}\langle f,e_k\rangle\otimes\langle g,e_k\rangle
=\int_{\mathbb R^d} f(q)\otimes g(q)\,dq,
\qquad f,g\in L^2(\mathbb R^d;V).
\end{equation}
Set, for $y\in\mathbb R^d$, $\mu\in\mathcal P(\mathbb R^d)$, $t\ge0$,
\begin{equation}\label{eq:b_k_def}
b_k(y,\mu,t):=\int_{\mathbb R^d} b(y,\mu,t,q)\,e_k(q)\,dq .
\end{equation}
%% NEW: b_k and the ONB are now defined BEFORE they are used.
By \eqref{eq:noise_expansion} the equation with interaction
\eqref{eq:interacting} is equivalent to
\begin{equation}\label{eq:SDE_component}
dx(u,t)=a\bigl(x(u,t),\mu_t,t\bigr)\,dt
+\sum_{k=1}^{\infty} b_k\bigl(x(u,t),\mu_t,t\bigr)\,d\beta_t^k ,
\end{equation}
and by \eqref{eq:Parseval}
\begin{equation}\label{eq:quadratic_variation_identity}
\sum_{k\ge1} b_k(y,\mu,t)\otimes b_k(y,\mu,t)
=\int_{\mathbb R^d} b(y,\mu,t,q)\otimes b(y,\mu,t,q)\,dq .
\end{equation}

For $\mu\in\mathcal P(\mathbb R^d)$ and $t\ge0$ define the elements of
$\mathcal H$
\begin{align}
A_t(\mu)&:=\int_{\mathbb R^d} Dh(y)\,a(y,\mu,t)\,\mu(dy),
\label{eq:A_def}\\
G_t(\mu)&:=\frac12\int_{\mathbb R^d}\!\!\int_{\mathbb R^d}
D^2h(y)\bigl(b(y,\mu,t,q),b(y,\mu,t,q)\bigr)\,dq\,\mu(dy),
\label{eq:G_t_def}\\
\Sigma_{t,k}(\mu)&:=\int_{\mathbb R^d} Dh(y)\,b_k(y,\mu,t)\,\mu(dy),
\qquad k\ge1 .
\label{eq:Sigma_def}
\end{align}

\paragraph{Assumption G.}
\begin{enumerate}
\item[\textup{(G1)}] $h\in C_b^2(\mathbb R^d;\mathcal H)$;
\item[\textup{(G2)}] $a$ is bounded, and $b$ is bounded with
$C_b:=\sup_{y,\mu,t}\int_{\mathbb R^d}\|b(y,\mu,t,q)\|^2\,dq<\infty$;
\item[\textup{(G3)}] the flow solving \eqref{eq:interacting}
exists, $\mu_t=\mu_0\circ x(\cdot,t)^{-1}$ is a Markov process, and
$h_t=\mathcal H(\mu_t)$ is Markov (Theorem~\ref{Markovian});
moreover the kernel $h$ is characteristic, i.e.\ $\mathcal H$ is injective
on $\mathcal P(\mathbb R^d)$ (Lemma~\ref{separates});

\item[\textup{(G4)}] $a$ and $b$ do not depend on $t$; $a$ is bounded and
continuous in $(y,\mu)$ (weak topology on $\mathcal P(\mathbb R^d)$), and
$(y,\mu)\mapsto b(y,\mu,\cdot)$ is continuous in $L^2(\mathbb R^d,dq)$,
uniformly in $y$.
\end{enumerate}

For $h\in\mathcal H(\mathcal P(\mathbb R^d))$ denote by $\mu(h)$ the unique
probability measure with $\mathcal H(\mu(h))=h$; it is well defined by
injectivity of $\mathcal H$. Since $\mathcal H$ is an injective Borel map
between the standard Borel spaces $\mathcal P(\mathbb R^d)$ and $\mathcal H$,
the Lusin--Souslin theorem implies that its range
$E:=\mathcal H(\mathcal P(\mathbb R^d))$ is Borel and that $\mu(\cdot)$ is
Borel measurable on $E$.

\begin{theorem}\label{thm:generator}
\leavevmode
\begin{enumerate}
\item[\textup{(a)}] Under \textup{(G1)--(G3)}, for every
$\Phi\in C_b^2(\mathcal H)$ the process
\begin{equation}\label{eq:martingale_form}
\Phi(h_t)-\Phi(h_0)-\int_0^t (\mathcal L_s\Phi)(h_s)\,ds,
\qquad t\ge0,
\end{equation}
is a square-integrable martingale with respect to $(\mathcal F_t)$, where,
for $h\in E$,
\begin{equation}\label{eq:generator_component}
(\mathcal L_s\Phi)(h)
=D\Phi(h)\bigl[A_s(\mu(h))+G_s(\mu(h))\bigr]
+\frac12\sum_{k=1}^{\infty}D^2\Phi(h)
\bigl(\Sigma_{s,k}(\mu(h)),\Sigma_{s,k}(\mu(h))\bigr).
\end{equation}
%% CHANGED: the drift now contains G_s -- the Ito correction of the map h.
%% This term was missing in the previous version of the statement.
\item[\textup{(b)}] If moreover \textup{(G4)} holds, then
$\mathcal L:=\mathcal L_s$ is the infinitesimal generator of the transition
semigroup $(P_t)$ of $(h_t)$ on $C_b^2(\mathcal H)$:
\begin{equation}\label{eq:generator_limit}
(\mathcal L\Phi)(h)=\lim_{t\downarrow0}\frac{P_t\Phi(h)-\Phi(h)}{t},
\qquad h\in E .
\end{equation}
%% CHANGED: time-inhomogeneous coefficients only give the martingale-problem
%% form (a); the single generator (b) is claimed under time-homogeneity (G4).
\end{enumerate}
\end{theorem}

\begin{proof}

\emph{Step 1 (It\^o formula for $h(x(u,t))$).}
Fix $u\in\mathbb R^d$. Since $h\in C_b^2(\mathbb R^d;\mathcal H)$ by
\textup{(G1)} and $x(u,t)$ is the It\^o process
\eqref{eq:SDE_component}, the It\^o formula yields
\begin{equation}\label{eq:Ito_component_expanded}
\begin{aligned}
d\,h(x(u,t))={}&Dh(x(u,t))\,a(x(u,t),\mu_t,t)\,dt
+\sum_{k=1}^{\infty}Dh(x(u,t))\,b_k(x(u,t),\mu_t,t)\,d\beta_t^k\\
&+\frac12\sum_{k=1}^{\infty}D^2h(x(u,t))
\bigl(b_k(x(u,t),\mu_t,t),\,b_k(x(u,t),\mu_t,t)\bigr)\,dt .
\end{aligned}
\end{equation}
Both series converge in $L^2(\Omega;\mathcal H)$, locally uniformly in $t$.
Indeed, by \eqref{eq:quadratic_variation_identity}, \textup{(G1)} and
\textup{(G2)},
\[
\sum_{k\ge1}\|Dh(y)\,b_k(y,\mu,t)\|^2
\le\|Dh\|_{\infty}^2\sum_{k\ge1}\|b_k(y,\mu,t)\|^2
=\|Dh\|_{\infty}^2\int_{\mathbb R^d}\|b(y,\mu,t,q)\|^2\,dq
\le\|Dh\|_{\infty}^2 C_b,
\]
and similarly
$\sum_{k\ge1}\|D^2h(y)(b_k,b_k)\|
\le\|D^2h\|_{\infty}\sum_{k\ge1}\|b_k(y,\mu,t)\|^2
\le\|D^2h\|_{\infty}C_b$.
%% NEW: the previously unproven convergence assertion is now justified.

\emph{Step 2 (averaging over the labels).}
Integrating \eqref{eq:Ito_component_expanded} with respect to $\mu_0(du)$,
the $dt$-terms are Bochner integrals and the summation may be interchanged
with $\mu_0(du)$ by dominated convergence, using the uniform bounds above.
For the stochastic part we invoke the stochastic Fubini theorem;
its square-integrability condition
\[
\int_{\mathbb R^d}\mathbb E\int_0^T
\sum_{k\ge1}\bigl\|Dh(x(u,t))\,b_k(x(u,t),\mu_t,t)\bigr\|^2\,dt\,\mu_0(du)
<\infty
\]
holds by the same estimate.
%% NEW: the use of stochastic Fubini is now explicit and justified.
The change of variables
$\int F(x(u,t))\,\mu_0(du)=\int F(y)\,\mu_t(dy)$ then gives the
$\mathcal H$-valued SDE
\begin{equation}\label{eq:ht_dynamics}
dh_t=\bigl[A_t(\mu_t)+G_t(\mu_t)\bigr]\,dt
+\sum_{k=1}^{\infty}\Sigma_{t,k}(\mu_t)\,d\beta_t^k .
\end{equation}
Note that $G_t$ is precisely the It\^o correction of the kernel map $h$: it
belongs to the \emph{drift} of $h_t$ and is kept in
\eqref{eq:ht_dynamics}.
%% CHANGED: this is the key point -- in the previous version the drift
%% carried G_t, but G_t was silently dropped when the Ito formula for
%% Phi(h_t) was applied two steps later.

The martingale part of \eqref{eq:ht_dynamics} admits the equivalent
white-noise form
\begin{equation}\label{eq:ht_noise_form}
\sum_{k=1}^{\infty}\Sigma_{t,k}(\mu_t)\,d\beta_t^k
=\int_{\mathbb R^d}\!\!\int_{\mathbb R^d}
Dh(y)\,b(y,\mu_t,t,q)\,\mu_t(dy)\,W(dt,dq),
\end{equation}
because the integrand
$f_t(q):=\int_{\mathbb R^d} Dh(y)\,b(y,\mu_t,t,q)\,\mu_t(dy)$ belongs to
$L^2(\mathbb R^d;\mathcal H)$:
\[
\int_{\mathbb R^d}\|f_t(q)\|^2\,dq
\le\int_{\mathbb R^d}\!\!\int_{\mathbb R^d}
\|Dh(y)\,b(y,\mu_t,t,q)\|^2\,\mu_t(dy)\,dq
\le\|Dh\|_{\infty}^2 C_b<\infty
\]
by Jensen's inequality ($\mu_t$ is a probability measure), and
$\langle f_t,e_k\rangle=\Sigma_{t,k}(\mu_t)$. Consequently
\begin{equation}\label{eq:Sigma_summability}
\sum_{k\ge1}\|\Sigma_{t,k}(\mu)\|^2
=\int_{\mathbb R^d}\|f_t(q)\|^2\,dq
\le\|Dh\|_{\infty}^2 C_b<\infty ,
\end{equation}
so the series in \eqref{eq:ht_dynamics} converges in
$L^2(\Omega;\mathcal H)$ and the quadratic variation of $h_t$ is
\begin{equation}\label{eq:ht_quadratic_variation}
d\langle h\rangle_t
=\sum_{k=1}^{\infty}\Sigma_{t,k}(\mu_t)\otimes\Sigma_{t,k}(\mu_t)\,dt .
\end{equation}

\emph{Step 3 (It\^o formula for $\Phi(h_t)$).}
Let $\Phi\in C_b^2(\mathcal H)$. Applying the It\^o formula to $\Phi$ at the It\^o process
\eqref{eq:ht_dynamics}, and using
\eqref{eq:ht_quadratic_variation}, we obtain
\begin{equation}\label{eq:Ito_Phi_expanded}
\begin{aligned}
d\,\Phi(h_t)={}&
D\Phi(h_t)\bigl[A_t(\mu_t)+G_t(\mu_t)\bigr]\,dt
+\sum_{k=1}^{\infty}D\Phi(h_t)\,\Sigma_{t,k}(\mu_t)\,d\beta_t^k\\
&+\frac12\sum_{k=1}^{\infty}D^2\Phi(h_t)
\bigl(\Sigma_{t,k}(\mu_t),\Sigma_{t,k}(\mu_t)\bigr)\,dt .
\end{aligned}
\end{equation}
%% CHANGED: D\Phi(h_t) now acts on the FULL drift A_t + G_t, consistently
%% with \eqref{eq:ht_dynamics}. Previously only A_t survived here.
All series converge in $L^2(\Omega)$: for the martingale part,
\[
\sum_{k\ge1}\mathbb E\int_0^T
\bigl|D\Phi(h_t)\,\Sigma_{t,k}(\mu_t)\bigr|^2\,dt
\le\|D\Phi\|_{\infty}^2
\int_0^T\sum_{k\ge1}\|\Sigma_{t,k}(\mu_t)\|^2\,dt
\le T\,\|D\Phi\|_{\infty}^2\|Dh\|_{\infty}^2C_b
\]
by \eqref{eq:Sigma_summability}, and for the second-order term,
\[
\sum_{k\ge1}\bigl|D^2\Phi(h_t)(\Sigma_{t,k},\Sigma_{t,k})\bigr|
\le\|D^2\Phi\|_{\infty}\sum_{k\ge1}\|\Sigma_{t,k}\|^2
\le\|D^2\Phi\|_{\infty}\|Dh\|_{\infty}^2 C_b<\infty .
\]
Since $D\Phi$ is bounded, the stochastic integral in
\eqref{eq:Ito_Phi_expanded} is a square-integrable martingale, not merely a
local one.

By \textup{(G3)} the embedding $\mathcal H$ is injective, hence
$\mu_t=\mu(h_t)$ pathwise, and the bounded drift integrand in
\eqref{eq:Ito_Phi_expanded} equals $(\mathcal L_t\Phi)(h_t)$. Integrating
\eqref{eq:Ito_Phi_expanded} from $0$ to $t$ proves part (a):
\begin{equation}\label{eq:Dynkin}
\Phi(h_t)=\Phi(h_0)+\int_0^t(\mathcal L_s\Phi)(h_s)\,ds+M_t^{\Phi},
\end{equation}
where $M^{\Phi}$ is a square-integrable martingale.

\emph{Step 4 (identification of the generator).}
Assume \textup{(G4)}, so the coefficients are time-homogeneous and $(h_t)$
is a time-homogeneous Markov process with transition semigroup
$P_t\Phi(h)=\mathbb E[\Phi(h_t)\mid h_0=h]$ (Theorem~\ref{Markovian}).
Taking the conditional expectation in \eqref{eq:Dynkin} with respect to
$\sigma(h_0)$ and using the Markov property together with Fubini's theorem,
\begin{equation}\label{eq:semigroup_integral}
P_t\Phi(h)-\Phi(h)=\int_0^t P_s(\mathcal L\Phi)(h)\,ds .
\end{equation}
By \textup{(G1)}, \textup{(G2)} and the estimates of Step 3,
\[
\|\mathcal L\Phi\|_{\infty}
\le\|D\Phi\|_{\infty}
\Bigl(\|Dh\|_{\infty}\|a\|_{\infty}
+\frac12\|D^2h\|_{\infty}C_b\Bigr)
+\frac12\|D^2\Phi\|_{\infty}\|Dh\|_{\infty}^2C_b=:C_{\Phi}<\infty .
\]
It remains to show that $s\mapsto P_s(\mathcal L\Phi)(h)$ is
right-continuous at $s=0$. The paths of $(h_t)$ are continuous, being the
sum of a bounded-variation part (bounded drift in \eqref{eq:ht_dynamics})
and a continuous martingale; hence $h_s\to h$ $\mathbb P_h$-a.s. Moreover
$\mu_s\to\mu(h)$ weakly, so by the bounded continuity of the coefficients
postulated in \textup{(G4)} and dominated convergence,
$A(\mu_s)\to A(\mu(h))$, $G(\mu_s)\to G(\mu(h))$ and
$\Sigma_k(\mu_s)\to\Sigma_k(\mu(h))$ in $\mathcal H$, for each $k$. Using
the uniform bound \eqref{eq:Sigma_summability} to dominate the tail of the
series, it follows that $(\mathcal L\Phi)(h_s)\to(\mathcal L\Phi)(h)$
$\mathbb P_h$-a.s.; since $\mathcal L\Phi$ is bounded, dominated convergence
gives $P_s(\mathcal L\Phi)(h)\to(\mathcal L\Phi)(h)$. Dividing
\eqref{eq:semigroup_integral} by $t$ and letting $t\downarrow0$ yields
\eqref{eq:generator_limit}.
\end{proof}

\subsection{Semigroup related to equation with interaction}

Throughout this subsection, the coefficients are assumed to be time-independent.
The equation with interaction is therefore written as
\begin{equation}
dx(u,t)
=
a(x(u,t),\mu_t)\,dt
+
\sum_{k=1}^{\infty}
b_k(x(u,t),\mu_t)\,d\beta_t^k,
\qquad
\mu_t=\mu_0\circ x(\cdot,t)^{-1}.
\end{equation}

Let
\begin{equation}
U=L^2(\mathbb R^d,dq),
\end{equation}
and let \(\{e_k\}_{k\geq1}\) be an orthonormal basis of \(U\). As before,
\begin{equation}
b_k(x,\mu)
=
\int_{\mathbb R^d}
b(x,\mu,q)e_k(q)\,dq.
\end{equation}

\begin{assumption}\label{ass:semigroup_basic}
Let \(K\subset\mathfrak M_2\) be compact with respect to the Wasserstein metric
\(\gamma_2\). Assume that the following conditions hold.

\begin{enumerate}
\item The set \(K\) is invariant under the measure-valued dynamics. Namely, for every
\(\mu_0\in K\),
\begin{equation}
\mu_t\in K
\qquad
\text{a.s. for every }t\geq0.
\end{equation}

\item The coefficients \(a\) and \(b\) are time-independent and satisfy the global
Lipschitz condition
\begin{equation}
\begin{split}
&
\|a(x,\mu)-a(y,\nu)\|
+
\left(
\int_{\mathbb R^d}
\|b(x,\mu,q)-b(y,\nu,q)\|^2\,dq
\right)^{1/2}
\leq
C\bigl(
\|x-y\|+\gamma_2(\mu,\nu)
\bigr),
\end{split}
\end{equation}
for all \(x,y\in\mathbb R^d\) and \(\mu,\nu\in\mathfrak M_2\).

\end{enumerate}
\end{assumption}

Set
\begin{equation}
G:=H(K)\subset\mathcal H.
\end{equation}
Since \(K\) is compact and \(H\) is continuous and injective, \(G\) is compact and
\begin{equation}
H|_K:K\longrightarrow G
\end{equation}
is a homeomorphism. In particular, \(H^{-1}:G\to K\) is continuous.

For \(\xi\in G\), let
\begin{equation}
\mu=H^{-1}(\xi).
\end{equation}
Define
\begin{equation}
\begin{split}
A(\xi)
=
\int_{\mathbb R^d}
Dh(y)a(y,\mu)\,\mu(dy)
+
\frac12
\int_{\mathbb R^d}
\int_{\mathbb R^d}
D^2h(y)
\bigl[
b(y,\mu,q),b(y,\mu,q)
\bigr]
\,dq\,\mu(dy),
\end{split}
\end{equation}
and
\begin{equation}
\Sigma_k(\xi)
=
\int_{\mathbb R^d}
Dh(y)b_k(y,\mu)\,\mu(dy),
\qquad k\geq1.
\end{equation}

Equivalently, define the operator
\(\mathscr S(\xi)\in\mathcal L_2(U,\mathcal H)\) by
\begin{equation}
\mathscr S(\xi)f
=
\int_{\mathbb R^d}
Dh(y)
\left[
\int_{\mathbb R^d}
b(y,\mu,q)f(q)\,dq
\right]
\mu(dy),
\qquad f\in U.
\end{equation}
Then
\begin{equation}
\mathscr S(\xi)e_k=\Sigma_k(\xi).
\end{equation}

For \(\Phi\in C_b^2(\mathcal H)\), define
\begin{equation}
(\widetilde{\mathcal L}_0\Phi)(\xi)
=
D\Phi(\xi)[A(\xi)]
+
\frac12
\sum_{k=1}^{\infty}
D^2\Phi(\xi)
\bigl[
\Sigma_k(\xi),\Sigma_k(\xi)
\bigr],
\qquad
\xi\in G.
\end{equation}

Let
\begin{equation}
\mathcal D_0
=
\left\{
\Phi|_G:\Phi\in C_b^2(\mathcal H)
\right\}.
\end{equation}

\begin{theorem}\label{thm:semigroup_derived}
Assume Assumption~\ref{ass:semigroup_basic}. For \(\varphi\in C(G)\), define
\begin{equation}
P_t\varphi(\xi)
=
\mathbb E_{\mu}
\left[
\varphi\bigl(H(\mu_t)\bigr)
\right],
\qquad
\mu=H^{-1}(\xi).
\end{equation}
%Then \(\{P_t\}_{t\geq0}\) is a strongly continuous contraction Markov semigroup on
%\(C(G)\).
Let \((\mathfrak L,D(\mathfrak L))\) be its infinitesimal generator. Then
\(\mathcal D_0\) is dense in \(C(G)\), and for every
\(\Phi\in C_b^2(\mathcal H)\),
\begin{equation}
\Phi|_G\in D(\mathfrak L)
\end{equation}
and
\begin{equation}
\mathfrak L(\Phi|_G)
=
\widetilde{\mathcal L}_0\Phi.
\end{equation}
Consequently, \(\widetilde{\mathcal L}_0\Phi\) depends only on the restriction
\(\Phi|_G\), and therefore defines a densely defined operator
\((\mathcal L_0,\mathcal D_0)\) on \(C(G)\). Moreover,
\begin{equation}
\mathcal L_0\subset\mathfrak L
\end{equation}
and
\begin{equation}
\overline{\mathcal L_0}\subset\mathfrak L.
\end{equation}

The generator \(\mathfrak L\) satisfies the Hille--Yosida conditions. In particular,
\(\mathfrak L\) is closed, densely defined, and dissipative, and for every
\(\lambda>0\),
\begin{equation}
\operatorname{Ran}(\lambda I-\mathfrak L)
=
C(G),
\end{equation}
with
\begin{equation}
\|(\lambda I-\mathfrak L)^{-1}\|_{\mathcal L(C(G))}
\leq
\frac1{\lambda}.
\end{equation}
\end{theorem}

\begin{proof}
Since \(K\) is compact in \((\mathfrak M_2,\gamma_2)\), its second moments are
uniformly bounded. Thus
\begin{equation}
M_K
:=
\sup_{\mu\in K}
\int_{\mathbb R^d}
\bigl(1+|y|^2\bigr)\,\mu(dy)
<
+\infty.
\end{equation}

The Lipschitz condition on \(a\) and \(b\) implies the linear growth estimate
\begin{equation}
|a(y,\mu)|
+
\left(
\int_{\mathbb R^d}
|b(y,\mu,q)|^2\,dq
\right)^{1/2}
\leq
C_1
\left(
1+|y|+
\left(
\int_{\mathbb R^d}|z|^2\,\mu(dz)
\right)^{1/2}
\right).
\end{equation}
Consequently, there exists \(C_K<+\infty\) such that
\begin{equation}
\sup_{\mu\in K}
\int_{\mathbb R^d}
\left(
|a(y,\mu)|^2
+
\int_{\mathbb R^d}
|b(y,\mu,q)|^2\,dq
\right)
\mu(dy)
\leq C_K.
\end{equation}

We first verify the summability of the lifted diffusion coefficients. By Jensen's
inequality and Parseval's identity,
\begin{equation}
\begin{split}
\sum_{k=1}^{\infty}
\|\Sigma_k(\xi)\|_{\mathcal H}^2
&=
\sum_{k=1}^{\infty}
\left\|
\int_{\mathbb R^d}
Dh(y)b_k(y,\mu)\,\mu(dy)
\right\|_{\mathcal H}^2
\\
&\leq
\int_{\mathbb R^d}
\sum_{k=1}^{\infty}
\|Dh(y)b_k(y,\mu)\|_{\mathcal H}^2
\,\mu(dy)
\\
&\leq
\|Dh\|_{\infty}^2
\int_{\mathbb R^d}
\int_{\mathbb R^d}
|b(y,\mu,q)|^2\,dq\,\mu(dy)
\\
&\leq
\|Dh\|_{\infty}^2 C_K.
\end{split}
\end{equation}
Hence
\begin{equation}
\sup_{\xi\in G}
\sum_{k=1}^{\infty}
\|\Sigma_k(\xi)\|_{\mathcal H}^2
<
+\infty.
\end{equation}

Similarly,
\begin{equation}
\begin{split}
&
\sum_{k=1}^{\infty}
\left\|
\int_{\mathbb R^d}
D^2h(y)
\bigl[
b_k(y,\mu),b_k(y,\mu)
\bigr]
\,\mu(dy)
\right\|_{\mathcal H}
\\
&\qquad\leq
\|D^2h\|_{\infty}
\int_{\mathbb R^d}
\sum_{k=1}^{\infty}
|b_k(y,\mu)|^2\,\mu(dy)
\\
&\qquad=
\|D^2h\|_{\infty}
\int_{\mathbb R^d}
\int_{\mathbb R^d}
|b(y,\mu,q)|^2\,dq\,\mu(dy)
\\
&\qquad\leq
\|D^2h\|_{\infty}C_K.
\end{split}
\end{equation}
Thus the second-order drift series also converges uniformly on \(G\).

The same estimates, together with the boundedness of \(Dh\), \(D^2h\),
and \(D^3h\), imply that there exists a constant \(C_K\) such that
\begin{equation}
\|A(H(\mu))-A(H(\nu))\|_{\mathcal H}
+
\|\mathscr S(H(\mu))-\mathscr S(H(\nu))\|_{\mathcal L_2(U,\mathcal H)}
\leq
C_K\gamma_2(\mu,\nu)
\end{equation}
for all \(\mu,\nu\in K\). Hence \(A\) is continuous on \(G\), and
\(\mathscr S:G\to\mathcal L_2(U,\mathcal H)\) is continuous. It follows that, for every
\(\Phi\in C_b^2(\mathcal H)\),
\begin{equation}
\widetilde{\mathcal L}_0\Phi\in C(G).
\end{equation}

We next prove the Feller property. Suppose that
\begin{equation}
\xi_n\longrightarrow\xi
\qquad
\text{in }G.
\end{equation}
Since \(H|_K\) is a homeomorphism,
\begin{equation}
\mu_n:=H^{-1}(\xi_n)
\longrightarrow
\mu:=H^{-1}(\xi)
\qquad
\text{in }\gamma_2.
\end{equation}

Using the same Wiener sheet for the solutions starting from \(\mu_n\) and \(\mu\),
one can see, for every
\(t\geq0\),
\begin{equation}
\mathbb E
\gamma_2(\mu_t^n,\mu_t)^2
\leq
C_t\gamma_2(\mu_n,\mu)^2.
\end{equation}
Since \(h\) is globally Lipschitz,
\begin{equation}
\mathbb E
\|H(\mu_t^n)-H(\mu_t)\|_{\mathcal H}^2
\leq
\|Dh\|_{\infty}^2
C_t\gamma_2(\mu_n,\mu)^2.
\end{equation}
Therefore,
\begin{equation}
H(\mu_t^n)
\longrightarrow
H(\mu_t)
\qquad
\text{in probability}.
\end{equation}
Every \(\varphi\in C(G)\) is bounded and uniformly continuous because \(G\) is
compact. Hence
\begin{equation}
P_t\varphi(\xi_n)
\longrightarrow
P_t\varphi(\xi).
\end{equation}
Thus
\begin{equation}
P_t\varphi\in C(G).
\end{equation}

The Markov property of the lifted process implies
\begin{equation}
P_{t+s}\varphi
=
P_tP_s\varphi,
\qquad
s,t\geq0,
\end{equation}
and
\begin{equation}
P_0\varphi=\varphi.
\end{equation}
Moreover, \(P_t\) is positive and preserves constants:
\begin{equation}
P_t\varphi\geq0
\qquad
\text{whenever }\varphi\geq0,
\end{equation}
and
\begin{equation}
P_t1=1.
\end{equation}
Consequently,
\begin{equation}
\|P_t\varphi\|_{C(G)}
\leq
\|\varphi\|_{C(G)}.
\end{equation}

It remains to prove strong continuity. Let the process start from \(\mu\in K\).
Using the stochastic equation, the Itô isometry, and the uniform coefficient estimate
obtained above, we obtain
\begin{equation}
\sup_{\mu\in K}
\mathbb E_{\mu}
\int_{\mathbb R^d}
|x(u,t)-u|^2\,\mu(du)
\leq
C_K(t+t^2).
\end{equation}
Since \(h\) is Lipschitz,
\begin{equation}
\sup_{\xi\in G}
\mathbb E_{\xi}
\|h_t-\xi\|_{\mathcal H}^2
\leq
\|Dh\|_{\infty}^2C_K(t+t^2).
\end{equation}

Let \(\omega_{\varphi}\) be the modulus of continuity of
\(\varphi\in C(G)\). For every \(\delta>0\),
\begin{equation}
\begin{split}
|P_t\varphi(\xi)-\varphi(\xi)|
&\leq
\omega_{\varphi}(\delta)
+
2\|\varphi\|_{C(G)}
\mathbb P_{\xi}
\left\{
\|h_t-\xi\|_{\mathcal H}>\delta
\right\}
\\
&\leq
\omega_{\varphi}(\delta)
+
\frac{
2\|\varphi\|_{C(G)}
}{
\delta^2
}
\mathbb E_{\xi}
\|h_t-\xi\|_{\mathcal H}^2.
\end{split}
\end{equation}
Taking the supremum over \(\xi\in G\) yields
\begin{equation}
\|P_t\varphi-\varphi\|_{C(G)}
\leq
\omega_{\varphi}(\delta)
+
\frac{
2\|\varphi\|_{C(G)}
\|Dh\|_{\infty}^2C_K
}{
\delta^2
}
(t+t^2).
\end{equation}
First letting \(t\downarrow0\), and then \(\delta\downarrow0\), gives
\begin{equation}
\lim_{t\downarrow0}
\|P_t\varphi-\varphi\|_{C(G)}
=
0.
\end{equation}
Thus \(\{P_t\}_{t\geq0}\) is a strongly continuous contraction Markov semigroup on
\(C(G)\).

We now prove that \(\mathcal D_0\) is dense in \(C(G)\). The class
\begin{equation}
\mathcal A
=
\left\{
\Phi|_G:\Phi\in C_b^2(\mathcal H)
\right\}
\end{equation}
is an algebra containing the constants. It also separates points of \(G\). Indeed, if
\(\xi,\eta\in G\) and \(\xi\neq\eta\), put
\begin{equation}
z=\xi-\eta.
\end{equation}
Choose \(\chi\in C_c^{\infty}(\mathbb R)\) such that
\(\chi(r)=r\) on an interval containing
\begin{equation}
\left\{
\langle\zeta,z\rangle_{\mathcal H}:\zeta\in G
\right\}.
\end{equation}
Then the function
\begin{equation}
\Phi(\zeta)
=
\chi\bigl(
\langle\zeta,z\rangle_{\mathcal H}
\bigr)
\end{equation}
belongs to \(C_b^2(\mathcal H)\), and
\begin{equation}
\Phi(\xi)-\Phi(\eta)
=
\|\xi-\eta\|_{\mathcal H}^2
\neq0.
\end{equation}
The Stone--Weierstrass theorem therefore gives
\begin{equation}
\overline{\mathcal D_0}^{\,C(G)}
=
C(G).
\end{equation}

Let now \(\Phi\in C_b^2(\mathcal H)\). Applying the Hilbert-valued Itô formula first
to \(h(x(u,t))\), and then to \(\Phi(h_t)\), yields
\begin{equation}
dh_t
=
A(h_t)\,dt
+
\sum_{k=1}^{\infty}
\Sigma_k(h_t)\,d\beta_t^k.
\end{equation}
Consequently,
\begin{equation}
\begin{split}
d\Phi(h_t)
={}&
D\Phi(h_t)[A(h_t)]\,dt
+
\sum_{k=1}^{\infty}
D\Phi(h_t)[\Sigma_k(h_t)]\,d\beta_t^k
\\
&+
\frac12
\sum_{k=1}^{\infty}
D^2\Phi(h_t)
\bigl[
\Sigma_k(h_t),\Sigma_k(h_t)
\bigr]
\,dt.
\end{split}
\end{equation}
Hence
\begin{equation}
d\Phi(h_t)
=
(\widetilde{\mathcal L}_0\Phi)(h_t)\,dt
+
\sum_{k=1}^{\infty}
D\Phi(h_t)[\Sigma_k(h_t)]\,d\beta_t^k.
\end{equation}

The uniform summability estimate gives
\begin{equation}
\sup_{\xi\in G}
\sum_{k=1}^{\infty}
\left|
D\Phi(\xi)[\Sigma_k(\xi)]
\right|^2
\leq
\|D\Phi\|_{\infty}^2
\sup_{\xi\in G}
\sum_{k=1}^{\infty}
\|\Sigma_k(\xi)\|_{\mathcal H}^2
<
+\infty.
\end{equation}
Thus
\begin{equation}
M_t^\Phi
:=
\Phi(h_t)-\Phi(h_0)
-
\int_0^t
(\widetilde{\mathcal L}_0\Phi)(h_s)\,ds
\end{equation}
is a square-integrable martingale.

Taking expectations for the process started from \(\xi\in G\), we obtain
\begin{equation}
P_t(\Phi|_G)(\xi)-\Phi(\xi)
=
\mathbb E_{\xi}
\int_0^t
(\widetilde{\mathcal L}_0\Phi)(h_s)\,ds.
\end{equation}
By the Markov property,
\begin{equation}
P_t(\Phi|_G)-\Phi|_G
=
\int_0^t
P_s(\widetilde{\mathcal L}_0\Phi)\,ds
\end{equation}
in \(C(G)\). Therefore,
\begin{equation}
\frac{
P_t(\Phi|_G)-\Phi|_G
}{t}
=
\frac1t
\int_0^t
P_s(\widetilde{\mathcal L}_0\Phi)\,ds.
\end{equation}
Since \(\widetilde{\mathcal L}_0\Phi\in C(G)\) and \(P_t\) is strongly continuous,
\begin{equation}
\lim_{t\downarrow0}
\left\|
\frac{
P_t(\Phi|_G)-\Phi|_G
}{t}
-
\widetilde{\mathcal L}_0\Phi
\right\|_{C(G)}
=
0.
\end{equation}
It follows that
\begin{equation}
\Phi|_G\in D(\mathfrak L)
\end{equation}
and
\begin{equation}
\mathfrak L(\Phi|_G)
=
\widetilde{\mathcal L}_0\Phi.
\end{equation}

In particular, if \(\Phi,\Psi\in C_b^2(\mathcal H)\) satisfy
\begin{equation}
\Phi|_G=\Psi|_G,
\end{equation}
then
\begin{equation}
\widetilde{\mathcal L}_0\Phi
=
\widetilde{\mathcal L}_0\Psi
\qquad
\text{on }G.
\end{equation}
Thus the formula defines an operator
\((\mathcal L_0,\mathcal D_0)\) on \(C(G)\), and
\begin{equation}
\mathcal L_0\subset\mathfrak L.
\end{equation}
Since \(\mathfrak L\) is closed,
\begin{equation}
\overline{\mathcal L_0}
\subset
\mathfrak L.
\end{equation}

Finally, since \(\mathfrak L\) is the generator of a strongly continuous contraction
semigroup, the Hille--Yosida theorem implies that it is closed, densely defined, and
dissipative, and that for every \(\lambda>0\),
\begin{equation}
\operatorname{Ran}(\lambda I-\mathfrak L)
=
C(G),
\end{equation}
and
\begin{equation}
\|(\lambda I-\mathfrak L)^{-1}\|_{\mathcal L(C(G))}
\leq
\frac1{\lambda}.
\end{equation}
This completes the proof.
\end{proof}

%The preceding theorem identifies the infinitesimal operator on a dense class of smooth test functions. To conclude that this class is a core, one needs an additional range or regularization condition.

\begin{corollary}\label{cor:semigroup_core_condition}
Assume the conditions of Theorem~\ref{thm:semigroup_derived}. Suppose, in addition,
that for some \(\lambda_0>0\),
\begin{equation}
\overline{
\operatorname{Ran}
(\lambda_0 I-\mathcal L_0)
}^{\,C(G)}
=
C(G).
\end{equation}
Then
\begin{equation}
\overline{\mathcal L_0}
=
\mathfrak L.
\end{equation}
Consequently, \(\mathcal D_0\) is a core for \(\mathfrak L\), and
\(\overline{\mathcal L_0}\) satisfies the Hille--Yosida conditions on \(C(G)\).
\end{corollary}

\begin{proof}
Since
\begin{equation}
\mathcal L_0\subset\mathfrak L
\end{equation}
and \(\mathfrak L\) generates a contraction semigroup, \(\mathcal L_0\) is
dissipative. Thus, for every \(\varphi\in\mathcal D_0\),
\begin{equation}
\|(\lambda_0 I-\mathcal L_0)\varphi\|_{C(G)}
\geq
\lambda_0\|\varphi\|_{C(G)}.
\end{equation}
The dense range condition implies that the closure
\(\overline{\mathcal L_0}\) is maximally dissipative. By the Lumer--Phillips theorem,
\(\overline{\mathcal L_0}\) generates a contraction semigroup. Since
\(\mathfrak L\) is a dissipative extension of \(\overline{\mathcal L_0}\), maximal
dissipativity gives
\begin{equation}
\overline{\mathcal L_0}
=
\mathfrak L.
\end{equation}
\end{proof}

\subsection{Dense range condition via  polynomials from measure}

The dense range condition can be analyzed through polynomial functionals
of the measure variable. Let
\begin{equation}
X=C_b^4(\mathbb R^d)
\end{equation}
be the Banach space of four times continuously differentiable functions
with bounded derivatives, equipped with the norm
\begin{equation}
\|f\|_X
=
\sum_{|\alpha|\leq4}
\sup_{x\in\mathbb R^d}
|D^\alpha f(x)|.
\end{equation}

Let \(K\) be a compact invariant set of probability measures and assume
that
\begin{equation}
G=H(K),
\end{equation}
where \(H:K\rightarrow G\)
is a homeomorphism.

Let \(\mathscr L\) denote the generator of the measure-valued Markov
process \((\mu_t)_{t\geq0}\) which is the solution to stochastic equation with interaction. The lifted process is defined by
\begin{equation}
h_t=H(\mu_t).
\end{equation}
We denote by \(\mathcal L_0\) the generator of the corresponding
Hilbert-valued Markov process. As it was calculated in (\ref{eq:generator_component}), for a sufficiently smooth function
\(\Phi:G\rightarrow\mathbb R\), it is given by
\begin{equation}
(\mathcal L_0\Phi)(h)
=
D\Phi(h)[A(h)]
+
\frac12
\sum_{k=1}^{\infty}
D^2\Phi(h)
[\Sigma_k(h),\Sigma_k(h)] .
\end{equation}

For a function \(P:K\rightarrow\mathbb R\), define its lift to \(G\) by
\begin{equation}
\widehat P(h)
=
P(H^{-1}(h)),
\qquad
h\in G .
\end{equation}
Then the generators of the original and lifted processes satisfy
\begin{equation}
(\mathcal L_0\widehat P)(H(\mu))
=
(\mathscr LP)(\mu),
\qquad
\mu\in K .
\end{equation}

\begin{proposition}
\label{prop:measure_polynomial_density}

Let \(K\) be compact in the weak topology. Consider the algebra
\(\mathcal P(K)\) generated by the linear functionals
\begin{equation}
\mu\longmapsto
\langle f,\mu\rangle
=
\int_{\mathbb R^d}f(x)\mu(dx),
\qquad
f\in C_b(\mathbb R^d).
\end{equation}
Then
\begin{equation}
\overline{\mathcal P(K)}^{\,C(K)}
=
C(K).
\end{equation}

Consequently, under the homeomorphism \(H:K\to G\),
\begin{equation}
\overline{\mathcal P_H(G)}^{\,C(G)}
=
C(G),
\end{equation}
where
\begin{equation}
\mathcal P_H(G)
=
\{P\circ H^{-1}:P\in\mathcal P(K)\}.
\end{equation}

\end{proposition}

\begin{proof}

The algebra \(\mathcal P(K)\) contains constants. Moreover, it separates
points of \(K\). Indeed, if \(\mu\neq\nu\), then there exists
\(f\in C_b(\mathbb R^d)\) such that
\begin{equation}
\langle f,\mu\rangle
\neq
\langle f,\nu\rangle .
\end{equation}
Therefore, the Stone--Weierstrass theorem implies
\begin{equation}
\overline{\mathcal P(K)}^{\,C(K)}
=
C(K).
\end{equation}

The second statement follows from the fact that \(H\) is a homeomorphism.

\end{proof}

\begin{proposition}
\label{prop:generator_polynomial}

Let
\begin{equation}
P(\mu)
=
\prod_{j=1}^{n}
\langle f_j,\mu\rangle ,
\qquad
f_j\in C_b^2(\mathbb R^d).
\end{equation}
Then the generator acting on \(P\) is given by
\begin{equation}
\begin{aligned}
(\mathscr LP)(\mu)
&=
\sum_{i=1}^{n}
\langle L^\mu f_i,\mu\rangle
\prod_{j\neq i}
\langle f_j,\mu\rangle
\\
&\quad+
\sum_{1\leq i<j\leq n}
\sum_{k=1}^{\infty}
\left\langle
\nabla f_i\cdot b_k(\cdot,\mu),
\mu
\right\rangle
\left\langle
\nabla f_j\cdot b_k(\cdot,\mu),
\mu
\right\rangle
\prod_{\ell\neq i,j}
\langle f_\ell,\mu\rangle ,
\end{aligned}
\end{equation}
where
\begin{equation}
L^\mu f(x)
=
a(x,\mu)\cdot\nabla f(x)
+
\frac12
\sum_{k=1}^{\infty}
D^2f(x)
[b_k(x,\mu),b_k(x,\mu)] .
\end{equation}

Assume furthermore that
\begin{equation}
a(x,\mu)
=
\sum_{\alpha=1}^{N_a}
a_\alpha(x)
\prod_{j=1}^{m_\alpha}
\langle\phi_{\alpha,j},\mu\rangle ,
\end{equation}
and
\begin{equation}
b_k(x,\mu)
=
\sum_{\beta=1}^{N_b}
b_{k,\beta}(x)
\prod_{j=1}^{n_\beta}
\langle\psi_{\beta,j},\mu\rangle ,
\end{equation}
where
\begin{equation}
a_\alpha,b_{k,\beta}\in C_b^2(\mathbb R^d),
\qquad
\phi_{\alpha,j},\psi_{\beta,j}\in C_b(\mathbb R^d).
\end{equation}

Define
\begin{equation}
d_a=\max_\alpha m_\alpha,
\qquad
d_b=\max_\beta n_\beta .
\end{equation}
Then
\begin{equation}
\mathscr L\mathcal P_n(K)
\subset
\mathcal P_{n+r}(K),
\end{equation}
where
\begin{equation}
r=\max\{d_a,2d_b\}.
\end{equation}

\end{proposition}

\begin{lemma}
\label{lem:projective_polynomial_norm}

Let \(\mathcal X\) be a Banach space continuously embedded into
\(C_b(\mathbb R^d)\). Namely, assume that there exists
\(C_{\mathcal X}>0\) such that
\begin{equation}
\|f\|_\infty
\leq
C_{\mathcal X}\|f\|_{\mathcal X},
\qquad
f\in\mathcal X .
\end{equation}

For \(n\geq1\), let
\(\mathcal P_n^{\mathrm{hom}}(K)\) be the space of homogeneous
 polynomials from measures of degree \(n\), consisting of finite linear
combinations of functionals
\begin{equation}
\mu\longmapsto
\prod_{j=1}^{n}
\langle f_j,\mu\rangle ,
\qquad
f_j\in\mathcal X .
\end{equation}

For \(P\in\mathcal P_n^{\mathrm{hom}}(K)\), define
\begin{equation}
|P|_n
=
\inf
\sum_{\ell=1}^{N}
|c_\ell|
\prod_{j=1}^{n}
\|f_{\ell,j}\|_{\mathcal X},
\end{equation}
where the infimum is taken over all representations
\begin{equation}
P(\mu)
=
\sum_{\ell=1}^{N}
c_\ell
\prod_{j=1}^{n}
\langle f_{\ell,j},\mu\rangle .
\end{equation}

Then \(|\cdot|_n\) defines a norm on
\(\mathcal P_n^{\mathrm{hom}}(K)\).

\end{lemma}

\begin{proof}

Non-negativity of \(|\cdot|_n\) follows directly from the definition. Homogeneity follows
from the fact that multiplying the coefficients of a representation by a
scalar gives a representation of the scaled polynomial.

For the triangle inequality, let
\(P,Q\in\mathcal P_n^{\mathrm{hom}}(K)\). Combining arbitrary
representations of \(P\) and \(Q\) gives a representation of \(P+Q\).
Taking the infimum over all representations yields
\begin{equation}
|P+Q|_n
\leq
|P|_n+|Q|_n .
\end{equation}

It remains to prove non-degeneracy. Assume that
\begin{equation}
|P|_n=0 .
\end{equation}
Then there exists a sequence of representations
\begin{equation}
P(\mu)
=
\sum_{\ell=1}^{N_m}
c_\ell^{(m)}
\prod_{j=1}^{n}
\langle f_{\ell,j}^{(m)},\mu\rangle
\end{equation}
such that
\begin{equation}
\sum_{\ell=1}^{N_m}
|c_\ell^{(m)}|
\prod_{j=1}^{n}
\|f_{\ell,j}^{(m)}\|_{\mathcal X}
\longrightarrow0 .
\end{equation}

For every probability measure \(\mu\in K\),
\begin{equation}
|\langle f,\mu\rangle|
\leq
\|f\|_\infty
\leq
C_{\mathcal X}\|f\|_{\mathcal X}.
\end{equation}
Therefore,
\begin{equation}
\begin{aligned}
|P(\mu)|
&\leq
C_{\mathcal X}^{\,n}
\sum_{\ell=1}^{N_m}
|c_\ell^{(m)}|
\prod_{j=1}^{n}
\|f_{\ell,j}^{(m)}\|_{\mathcal X}.
\end{aligned}
\end{equation}
Letting \(m\to\infty\), we obtain
\begin{equation}
P(\mu)=0,
\qquad
\mu\in K .
\end{equation}
Hence \(P=0\), which proves that \(|\cdot|_n\) is a norm.

\end{proof}

\begin{definition}
\label{def:weighted_polynomial_space}

Let
\begin{equation}
P=P_0+P_1+\cdots+P_N,
\qquad
P_n\in\mathcal P_n^{\mathrm{hom}}(K).
\end{equation}

Let \((\omega_n)_{n\geq0}\) be a sequence of positive weights.
Define
\begin{equation}
\|P\|_\rho
=
\sum_{n=0}^{N}
\omega_n|P_n|_n .
\end{equation}

The weighted polynomial space \(\mathcal B_\rho\) is defined as the
completion of \(\mathcal P(K)\) with respect to this norm. Let's consider the weights such that the embedding
\begin{equation}
\mathcal B_\rho\hookrightarrow C(K)
\end{equation}
is continuous. For example

Equivalently, we assume that there exists \(C_\rho>0\) such that
\begin{equation}
\|P\|_{C(K)}
\leq
C_\rho\|P\|_\rho ,
\qquad
P\in\mathcal P(K).
\end{equation}

\end{definition}

\begin{proposition}
\label{prop:generator_weighted_bounded}

Assume that the generator satisfies
\begin{equation}
\mathscr L
\mathcal P_n^{\mathrm{hom}}(K)
\subset
\bigoplus_{m\leq n+r}
\mathcal P_m^{\mathrm{hom}}(K)
\end{equation}
and that there exist constants \(A_n>0\) such that
\begin{equation}
|\mathscr LP|_{n+r}
\leq
A_n|P|_n,
\qquad
P\in\mathcal P_n^{\mathrm{hom}}(K).
\end{equation}

Suppose that the weights satisfy
\begin{equation}
\sup_{n\geq0}
\frac{\omega_{n+r}}{\omega_n}
A_n
<+\infty .
\end{equation}

Then there exists \(M_\rho>0\) such that
\begin{equation}
\|\mathscr LP\|_\rho
\leq
M_\rho\|P\|_\rho ,
\qquad
P\in\mathcal P(K).
\end{equation}

Consequently, \(\mathscr L\) extends to a bounded operator on
\(\mathcal B_\rho\).

\end{proposition}

\begin{proof}

Write
\begin{equation}
P=\sum_nP_n .
\end{equation}
By the degree estimate,
\(\mathscr LP_n\) only contains homogeneous components of degree at most
\(n+r\). Hence, using the definition of the weighted norm,
\begin{equation}
\|\mathscr LP\|_\rho
\leq
\sum_n
\omega_{n+r}
|\mathscr LP_n|_{n+r}.
\end{equation}

Using the assumption on the generator,
\begin{equation}
\|\mathscr LP\|_\rho
\leq
\sum_n
\omega_{n+r}A_n|P_n|_n .
\end{equation}

Therefore,
\begin{equation}
\|\mathscr LP\|_\rho
\leq
M_\rho
\sum_n
\omega_n|P_n|_n
=
M_\rho\|P\|_\rho .
\end{equation}

\end{proof}

\begin{theorem}[Dense range condition]
\label{thm:dense_range_condition}

Assume that the weighted polynomial space
\(\mathcal B_\rho\) satisfies the assumptions of
Definition~\ref{def:weighted_polynomial_space} and that
\(\mathscr L\) extends to a bounded operator on \(\mathcal B_\rho\).

Let
\begin{equation}
\lambda_0>M_\rho .
\end{equation}

Then
\begin{equation}
\overline{
\operatorname{Ran}
(\lambda_0I-\mathscr L)
}^{\,C(K)}
=
C(K).
\end{equation}

Consequently,
\begin{equation}
\overline{
\operatorname{Ran}
(\lambda_0I-\mathcal L_0)
}^{\,C(G)}
=
C(G).
\end{equation}

\end{theorem}

\begin{proof}

For \(P\in\mathcal P(K)\), define
\begin{equation}
u
=
\sum_{m=0}^{\infty}
\lambda_0^{-m-1}
\mathscr L^mP .
\end{equation}

Since
\(\mathscr L\) is bounded on \(\mathcal B_\rho\),
\begin{equation}
\|\mathscr L^mP\|_\rho
\leq
M_\rho^m\|P\|_\rho .
\end{equation}

Because
\(\lambda_0>M_\rho\), the series converges in
\(\mathcal B_\rho\).

For the partial sums
\begin{equation}
u_N
=
\sum_{m=0}^{N}
\lambda_0^{-m-1}
\mathscr L^mP ,
\end{equation}
we have
\begin{equation}
(\lambda_0I-\mathscr L)u_N
=
P-
\lambda_0^{-N-1}
\mathscr L^{N+1}P .
\end{equation}

Moreover,
\begin{equation}
\left\|
\lambda_0^{-N-1}
\mathscr L^{N+1}P
\right\|_\rho
\leq
\left(
\frac{M_\rho}{\lambda_0}
\right)^{N+1}
\|P\|_\rho
\rightarrow0 .
\end{equation}

The continuous embedding
\(\mathcal B_\rho\hookrightarrow C(K)\) gives
\begin{equation}
\lambda_0^{-N-1}
\mathscr L^{N+1}P
\rightarrow0
\quad\text{in }C(K).
\end{equation}

Hence
\begin{equation}
P
\in
\overline{
\operatorname{Ran}
(\lambda_0I-\mathscr L)
}^{\,C(K)} .
\end{equation}

By Proposition~\ref{prop:measure_polynomial_density},
 polynomials from measures are dense in \(C(K)\). Therefore,
\begin{equation}
\overline{
\operatorname{Ran}
(\lambda_0I-\mathscr L)
}^{\,C(K)}
=
C(K).
\end{equation}

The final assertion follows from the homeomorphism \(H:K\to G\).

\end{proof}

\section{Definition of Stability and sufficient condition}\label{sec:stability}
We refer to a definition in \cite{Dorogovtsev_2024}.
\begin{definition}
A set of random measures $\{\mu_\alpha; \alpha \in \mathcal{A}\} \subset \mathfrak{M}_n$ is called shift compact if for every $\alpha \in \mathcal{A}$ there is a random vector $u_\alpha \in \mathbb{R}^d$ such that $\{\mu_\alpha - u_\alpha, \alpha \in \mathcal{A}\}$ is weakly compact in $\mathfrak{M}_n$, i.e., for every $\varepsilon > 0$ there is a compact set $F \subset \mathfrak{M}_n$ for which the following relation holds:
\begin{equation}
\forall \alpha \in \mathcal{A} \quad P(\mu_\alpha - u_\alpha \in F) > 1 - \varepsilon.
\end{equation}
\end{definition}
\subsection{Asymptotic Behavior}
\begin{theorem}\label{cov}
Assume that in 
\begin{equation}
dx(u,t)=\int_{\mathbb{R}^d}\varphi(x(u,t)-v)\mu_t(dv)dt
+\int_{\mathbb{R}^d}b(x(u,t),\mu_t,q)W(dq,dt)
\end{equation}
$\mu_0 \in \mathfrak{M}_{2n+2}$, $\varphi$ satisfies the condition
\begin{equation}
(u - v, \varphi(u) - \varphi(v)) \leq -\alpha \|u - v\|^2, \quad u, v \in \mathbb{R}^d,
\end{equation}
$b$ satisfies the Lipschitz condition relatively to both variables with constant $B$ and
\begin{equation}
\alpha - \frac{1}{2} B^2 (2n + 1) \geq 0.
\end{equation}

Then for any fixed $h\in H$ we have
\begin{equation}
\begin{split}
&COV((g_{t+k},h),(g_{s+k},h))
\to COV_\infty(t-s),\quad k\to\infty
\end{split}
\end{equation}
where
\begin{equation}\label{eq:cov_infty_def}
\begin{split}
\mathrm{COV}_\infty(\tau)
=&(g,h(0))\cdot\left[\mathbb{E}(g,h(x(0,\tau)))-(g,h(0))\right].
\end{split}
\end{equation}
\end{theorem}
\begin{proof}
The set $\{\mu_t; t \geq 0\}$ is shift compact in $\mathfrak{M}_{2n}$ by theorem 5.2.1 in \cite{Dorogovtsev_2024}.

Denote
\begin{equation}
m_t=\int_{\mathbb{R}^d}x(u,t)\mu_0(\mathrm{d}u).
\end{equation}

Step 1. We claim that $\{\mu_t-m_t\}$ satisfies the Cauchy criteria, namely

$\forall \varepsilon>0,\exists\  T>0,$ s.t.
\begin{equation}
 \gamma_{2n}(\mu_t-m_t,\mu_s-m_s)<\varepsilon,\forall s,t>T
\end{equation}

To prove it, consider
\begin{equation}
\begin{split}
&\mathbb{E}\gamma_{2n}(\mu_t-m_t,\mu_s-m_s)\\
=&\mathbb{E}\left(\underset{\varkappa\in C(\mu,\nu)}{\inf}\iint_{\mathbb{R}^d\times\mathbb{R}^d}\|u-v\|^{2n}\varkappa(\mathrm{d}u,\mathrm{d}v)\right)^\frac{1}{2n}\\
\le&\mathbb{E}\left(\iint_{\mathbb{R}^d\times\mathbb{R}^d}\|u-v\|^{2n}(\mu_t-m_t)(\mathrm{d}u)(\mu_s-m_s)(\mathrm{d}v)\right)^\frac{1}{2n}\\
=&\mathbb{E}\left(\iint_{\mathbb{R}^d\times\mathbb{R}^d}\|x(u,t)-m_t-x(v,s)+m_s\|^{2n}\mu_0(\mathrm{d}u)\mu_0(\mathrm{d}v)\right)^\frac{1}{2n}\\
\le&\iint_{\mathbb{R}^d\times\mathbb{R}^d}\left(\mathbb{E}\|x(u,t)-m_t-x(v,s)+m_s\|^{2n}\mu_0(\mathrm{d}u)\mu_0(\mathrm{d}v)\right)^\frac{1}{2n}\\
\end{split}
\end{equation}
Thus, we need to estimate $\|x(u,t)-m_t\|$

Step 2. Define the difference $\Delta(t):=x(u,t) - x(v,t)$. For notational simplicity, suppose that $\Delta^k(t)$ is the $k_{\text{th}}$ component of $\Delta(t)$ and $\Delta_b^k(t):=b^k(x(u,s),\mu_s,q) - b^k(x(v,s),\mu_s,q)$. 
We aim to compute  $\mathbb{E}\|\Delta(t)\|^{2n}$ by Ito's formula.
\begin{equation}
\begin{split}
&\mathbb{E} \|\Delta(t)\|^{2n}\\
=&\|u-v\|^{2n}+ 2n\int_0^t\mathbb{E} \left[ \|\Delta(s)\|^{2n - 2}\cdot\left( \Delta(s), \int_{\mathbb{R}^d} (\varphi(x(u,s) - y) - \varphi(x(v,s) - y)) \mu_s(\mathrm{d}y) \right) \right]\mathrm{d}s \\
&+ n\int_0^t\mathbb{E} \left[ \|\Delta(s)\|^{2n - 4}\sum_{k=1}^d [(2n-2)(\Delta^k(s))^2+\|\Delta(s)\|^{2}]\left(\Delta_b^k(s)\right)^2 \right]\mathrm{d}s\\
\le&\|u-v\|^{2n}-2n\alpha\int_0^t\mathbb{E}\|\Delta(s)\|^{2n}\mathrm{d}s+n(2n-1)B^2\int_0^t\mathbb{E}\|\Delta(s)\|^{2n}\mathrm{d}s\\
=&\|u-v\|^{2n}-n\left(2\alpha - (2n - 1) B^2\right)\int_0^t\mathbb{E}\|\Delta(s)\|^{2n}\mathrm{d}s
\end{split}
\end{equation}
where $2\alpha - (2n - 1) B^2>0$ by condition.

By Grönwall’s inequality, yields
\begin{equation}
\mathbb{E} \|\Delta(t)\|^{2n} \le \|u-v\|^{2n} e^{-n\left(2\alpha - (2n - 1) B^2\right) t}
\end{equation}
which implies  all the particles gradually stick together and move as one heavy particle. Therefore, under such consideration, when $t\to\infty$ we have
\begin{equation}
\begin{split}
&\mathbb{E}\|x(u,t)-m_t\|^{2n}\\
=&\mathbb{E}\left\|x(u,t)-\int_{\mathbb{R}^d}x(w,t)\mu_0(\mathrm{d}w)\right\|^{2n}\\
=&\mathbb{E}\left\|\int_{\mathbb{R}^d}(x(u,t)-x(w,t))\mu_0(\mathrm{d}w)\right\|^{2n}\\
\le&\int_{\mathbb{R}^d}\mathbb{E}\|x(u,t)-x(w,t)\|^{2n}\mu_0(\mathrm{d}w)\\
\le&\int_{\mathbb{R}^d}\|u-w\|^{2n}e^{-n\left(2\alpha - (2n - 1) B^2\right) t}\mu_0(\mathrm{d}w)\\
\le&M\cdot e^{-n\left(2\alpha - (2n - 1) B^2\right) t}\to  0 ,\quad t\to\infty
\end{split}
\end{equation}
where $M<+\infty$ is a upper bound for $\int_{\mathbb{R}^d}\|u-w\|^{2n}\mu_0(\mathrm{d}w)$. The existence of $M$ is obtained by $\mu_0\in\mathfrak{M}_{2n+2}$
Thus, by dominated convergence theorem, $\gamma_{2n}(\mu_t-m_t,\mu_s-m_s)\to0$ and finally we obtain the existence of $\mu_t-m_t$ when $t\to\infty$.

In conclusion,
\begin{equation}
\gamma_{2n}(\mu_t-m_t,\mu_s-m_s)\to0,\quad \text{when }s,t\to\infty
\end{equation}
by Cauchy criteria, $\lim_{t\to\infty}(\mu_t-m_t)$ exists, denoted by $\nu_{\infty}$ 

Step 3. Consider the covariance.
Recall that
\begin{equation}\label{eq:ht_def}
h_t=\int_{\mathbb{R}^d}h(x(u,t))\,\mu_0(\mathrm{d}u),
\qquad
(g,h_t)=\int_{\mathbb{R}^d}(g,h(x(u,t)))\,\mu_0(\mathrm{d}u).
\end{equation}
Define the scalar function
\begin{equation}\label{eq:Phi_def}
\Phi(x):=(g,h(x)), \qquad x\in\mathbb{R}^d.
\end{equation}
By the assumptions on $h$ and $g$, we have $\Phi\in C_b^2(\mathbb{R}^d)$.

We consider the covariance
\begin{equation}\label{eq:cov_start}
\begin{split}
\mathrm{COV}\big((g,h_{t+k}),(g,h_{s+k})\big)
&=
\mathrm{COV}\Big(
\int_{\mathbb{R}^d}\Phi(x(u,t+k))\,\mu_0(\mathrm{d}u),
\int_{\mathbb{R}^d}\Phi(x(v,s+k))\,\mu_0(\mathrm{d}v)
\Big)\\
&=\iint_{\mathbb{R}^d\times\mathbb{R}^d}
\mathrm{COV}\big(\Phi(x(u,t+k)),\Phi(x(v,s+k))\big)
\,\mu_0(\mathrm{d}u)\mu_0(\mathrm{d}v).
\end{split}
\end{equation}

We decompose
\begin{equation}\label{eq:decompose_x}
x(u,t+k)=\big(x(u,t+k)-m_{t+k}\big)+m_{t+k},
\end{equation}
where
\begin{equation}\label{eq:mt_def}
m_t=\int_{\mathbb{R}^d}x(u,t)\,\mu_0(\mathrm{d}u).
\end{equation}

By Step~2, for any $n\ge1$,
\begin{equation}\label{eq:collapse}
\mathbb{E}\|x(u,t+k)-m_{t+k}\|^{2n}\xrightarrow[k\to\infty]{}0,
\end{equation}
uniformly in $u$ with respect to $\mu_0$.

Since $\Phi\in C_b^1(\mathbb{R}^d)$, by Taylor’s formula and Jensen’s inequality,
\begin{equation}\label{eq:Phi_collapse}
\mathbb{E}\big|\Phi(x(u,t+k))-\Phi(m_{t+k})\big|
\xrightarrow[k\to\infty]{}0,
\end{equation}
and similarly,
\begin{equation}\label{eq:Phi_collapse_2}
\mathbb{E}\big|\Phi(x(v,s+k))-\Phi(m_{s+k})\big|
\xrightarrow[k\to\infty]{}0.
\end{equation}

Therefore,
\begin{equation}\label{eq:cov_reduction}
\mathrm{COV}\big(\Phi(x(u,t+k)),\Phi(x(v,s+k))\big)
-
\mathrm{COV}\big(\Phi(m_{t+k}),\Phi(m_{s+k})\big)
\xrightarrow[k\to\infty]{}0,
\end{equation}
uniformly in $u,v$. By dominated convergence, we obtain
\begin{equation}\label{eq:cov_limit_mt}
\mathrm{COV}\big((g,h_{t+k}),(g,h_{s+k})\big)
\xrightarrow[k\to\infty]{}
\mathrm{COV}\big(\Phi(m_{t+k}),\Phi(m_{s+k})\big).
\end{equation}

Next, we analyze the difference $m_{t+k}-m_{s+k}$. We compute
\begin{equation}\label{eq:mt_difference}
\begin{aligned}
m_{t+k}-m_{s+k}
&=\int_{\mathbb{R}^d}\big(x(u,t+k)-x(u,s+k)\big)\,\mu_0(\mathrm{d}u)\\
&=\int_{\mathbb{R}^d}\big(x(u,t-s)-u\big)\,\mu_{s+k}(\mathrm{d}u).
\end{aligned}
\end{equation}

By Step~1, $\mu_{s+k}-m_{s+k}$ converges in $\gamma_{2n}$ to $\nu_\infty$. Hence,
\begin{equation}\label{eq:mt_limit}
m_{t+k}-m_{s+k}
=
\int_{\mathbb{R}^d}x(u,t-s)\,(\mu_{s+k}-m_{s+k})(\mathrm{d}u)
\xrightarrow[k\to\infty]{}
\int_{\mathbb{R}^d}x(u,t-s)\,\nu_\infty(\mathrm{d}u).
\end{equation}

Consequently, the finite-dimensional distributions of $\{m_{t+k}\}_{t\ge0}$ converge to those of a stationary process whose increments depend only on the time difference $t-s$. Combining this with \eqref{eq:cov_limit_mt}, we conclude that
\begin{equation}\label{eq:final_cov}
\lim_{k\to\infty}
\mathrm{COV}\big((g,h_{t+k}),(g,h_{s+k})\big)
=
\mathrm{COV}_\infty(t-s),
\end{equation}
where
\begin{equation}
\begin{split}
\mathrm{COV}_\infty(t-s)
:=&
\mathrm{COV}\Big(
\Phi\Big(\int_{\mathbb{R}^d}x(u,t-s)\,\nu_\infty(\mathrm{d}u)\Big),
\Phi\Big(\int_{\mathbb{R}^d}x(u,0)\,\nu_\infty(\mathrm{d}u)\Big)
\Big)\\
=&
\mathrm{COV}\Big(
\Phi\Big(\int_{\mathbb{R}^d}x(u,t-s)\,\nu_\infty(\mathrm{d}u)\Big),
\Phi\Big(\int_{\mathbb{R}^d}u\,\nu_\infty(\mathrm{d}u)\Big)
\Big).
\end{split}
\end{equation}

\end{proof}

\subsection{Basin for attraction}
Fix a probability measure \(\nu_0\), and define
\begin{equation}
\Phi(u):=(h(u),\mathrm{H}(\nu_0))_H=\int_{\mathbb R^d}\Gamma(u-v)\,\nu_0(dv).
\end{equation}

Define
\begin{equation}
F(\mu):=\int_{\mathbb R^d}\Phi(u)\,\mu(du)-c,\quad Q(\mu):=\int_{\mathbb R^d}|\nabla\Phi(u)|^2\,\mu(du),\quad D(\mu):=\int_{\mathbb R^d}\Delta\Phi(u)\,\mu(du).
\end{equation}

Let
\begin{equation}
\mathfrak{M}_\Phi
:=
\left\{
\mu\in\mathfrak{M}(\mathbb R^d):
Q(\mu)>0
\right\}.
\end{equation}

For \(\kappa_0>0\), set
\begin{equation}
\kappa(t)=\kappa_0+t.
\end{equation}
On \(\mathcal P_\Phi\), define the feedback drift
\begin{equation}\label{eq:feedback_drift}
a(u,\mu,t)
:=
-
\frac{
\kappa(t)F(\mu)+\frac{\sigma^2}{2}D(\mu)
}{
Q(\mu)
}
\nabla\Phi(u).
\end{equation}
We consider the stochastic flow
\begin{equation}\label{eq:feedback_flow}
dx(u,t)
=
a(x(u,t),\mu_t,t)\,dt
+
\sigma\,dB_t,
\qquad
x(u,0)=u.
\end{equation}

For \(R>0\), define the basin
\begin{equation}
\mathcal B_R
:=
\left\{
\mu\in\mathfrak M_\Phi:
|F(\mu)|\le R
\right\}.
\end{equation}

\begin{theorem}\label{thm:feedback_exact_attraction}
Assume that \(\Phi\in C_b^2(\mathbb R^d)\). Let \((\mu_t)_{t\ge0}\) be a solution of \eqref{eq:feedback_flow} such that
\begin{equation}
\mu_t\in\mathfrak M_\Phi,
\qquad t\ge0.
\end{equation}
Then, for every initial measure \(\mu_0\in\mathcal B_R\), one has as $t\to\infty$
\begin{equation}
F(\mu_t)\stackrel{L^2}\longrightarrow0.
\end{equation}
% in \(L^2\) as \(t\to\infty\). More precisely,
% \begin{equation}\label{eq:L2_decay_feedback}
% \mathbb E|F(\mu_t)|^2
% \le
% e^{-2K(t)}|F(\mu_0)|^2
% +
% \frac{\sigma^2\|\nabla\Phi\|_\infty^2}{2\kappa_0+t},
% \end{equation}
% where
% \begin{equation}
% K(t)=\int_0^t\kappa(s)\,ds
% =
% \kappa_0t+\frac{t^2}{2}.
% \end{equation}
% In particular,
% \begin{equation}
% \lim_{t\to\infty}\mathbb E|F(\mu_t)|^2=0.
% \end{equation}
% Consequently,
% \begin{equation}
% F(\mu_t)\to0
% \end{equation}
% in probability. Moreover, \(F(\mu_t)\to0\) almost surely.

% Equivalently,
% \begin{equation}
% (m(\mu_t),m(\nu_0))_H\longrightarrow c.
% \end{equation}
% Thus the lifted trajectory is attracted to the constraint set \(M\) in the direction determined by \(m(\nu_0)\).
\end{theorem}

\begin{proof}
By definition,
\begin{equation}
F(\mu_t)
=
\int_{\mathbb R^d}\Phi(x(u,t))\,\mu_0(du)-c.
\end{equation}
Since \(\Phi\in C_b^2(\mathbb R^d)\), It\^o's formula gives
\begin{equation}
d\Phi(x(u,t))
=
\left[
a(x(u,t),\mu_t,t)\cdot\nabla\Phi(x(u,t))
+
\frac{\sigma^2}{2}\Delta\Phi(x(u,t))
\right]dt
+
\sigma\nabla\Phi(x(u,t))\cdot dB_t.
\end{equation}
Integrating with respect to \(\mu_0(du)\), and using
\begin{equation}
\mu_t=\mu_0\circ x(\cdot,t)^{-1},
\end{equation}
we obtain
\begin{equation}
dF(\mu_t)
=
A_t\,dt+\sigma G_t\cdot dB_t,
\end{equation}
where
\begin{equation}
A_t
=
\int_{\mathbb R^d}
\left[
a(y,\mu_t,t)\cdot\nabla\Phi(y)
+
\frac{\sigma^2}{2}\Delta\Phi(y)
\right]\mu_t(dy),
\end{equation}
and
\begin{equation}
G_t
=
\int_{\mathbb R^d}\nabla\Phi(y)\,\mu_t(dy).
\end{equation}

We now compute \(A_t\). By \eqref{eq:feedback_drift},
\begin{equation}
a(y,\mu_t,t)\cdot\nabla\Phi(y)
=
-
\frac{
\kappa(t)F(\mu_t)+\frac{\sigma^2}{2}D(\mu_t)
}{
Q(\mu_t)
}
|\nabla\Phi(y)|^2.
\end{equation}
Hence
\begin{equation}
\int_{\mathbb R^d}
a(y,\mu_t,t)\cdot\nabla\Phi(y)\,\mu_t(dy)
=
-\kappa(t)F(\mu_t)-\frac{\sigma^2}{2}D(\mu_t).
\end{equation}
Since
\begin{equation}
D(\mu_t)
=
\int_{\mathbb R^d}\Delta\Phi(y)\,\mu_t(dy),
\end{equation}
we get
\begin{equation}
A_t
=
-\kappa(t)F(\mu_t).
\end{equation}
Therefore the scalar process \(F_t:=F(\mu_t)\) satisfies
\begin{equation}\label{eq:F_linear_feedback}
dF_t
=
-\kappa(t)F_t\,dt
+
\sigma G_t\cdot dB_t.
\end{equation}

Let
\begin{equation}
K(t)=\int_0^t\kappa(s)\,ds=t\kappa_0+\frac{t^2}{2}.
\end{equation}
Solving the linear stochastic equation \eqref{eq:F_linear_feedback}, we obtain
\begin{equation}
F_t
=
e^{-K(t)}F_0
+
\sigma
\int_0^t
e^{-(K(t)-K(s))}G_s\cdot dB_s.
\end{equation}
Since
\begin{equation}
|G_s|
=
\left|
\int_{\mathbb R^d}\nabla\Phi(y)\,\mu_s(dy)
\right|
\le
\|\nabla\Phi\|_\infty,
\end{equation}
It\^o's isometry implies
\begin{equation}
\mathbb E|F_t|^2
\le
e^{-2K(t)}|F_0|^2
+
\sigma^2\|\nabla\Phi\|_\infty^2
\int_0^t
e^{-2(K(t)-K(s))}\,ds.
\end{equation}
For \(\kappa(t)=\kappa_0+t\), we have
\begin{equation}
K(t)-K(s)
=
\kappa_0(t-s)+\frac{t^2-s^2}{2}.
\end{equation}
Putting \(r=t-s\), we obtain
\begin{equation}
K(t)-K(t-r)
=
\kappa_0r+tr-\frac{r^2}{2}.
\end{equation}
For \(0\le r\le t\),
\begin{equation}
tr-\frac{r^2}{2}\ge \frac{t}{2}r.
\end{equation}
Therefore
\begin{equation}
K(t)-K(t-r)
\ge
\left(\kappa_0+\frac{t}{2}\right)r.
\end{equation}
It follows that
\begin{equation}
\int_0^t
e^{-2(K(t)-K(s))}\,ds
\le
\int_0^t
e^{-2(\kappa_0+\frac{t}{2})(t-s)}\,ds
\le
\frac{1}{2\kappa_0+t}.
\end{equation}
Consequently,
\begin{equation}
\mathbb E|F_t|^2
\le
e^{-2K(t)}|F_0|^2
+
\frac{\sigma^2\|\nabla\Phi\|_\infty^2}{2\kappa_0+t}.
\end{equation}
In particular,
\begin{equation}
\lim_{t\to\infty}\mathbb E|F_t|^2=0.
\end{equation}

% We now prove almost sure convergence. Fix \(p>2\). By the Burkholder--Davis--Gundy inequality applied to the stochastic convolution in \eqref{eq:F_linear_feedback}, there exists a constant \(C_p>0\) such that, for every integer \(n\ge1\),
% \begin{equation}
% \mathbb E\left[
% \sup_{n\le t\le n+1}|F_t|^p
% \right]
% \le
% C_p n^{-p/2}.
% \end{equation}
% Since \(p>2\), the series
% \begin{equation}
% \sum_{n=1}^\infty n^{-p/2}
% \end{equation}
% converges. Hence, by the Borel--Cantelli lemma,
% \begin{equation}
% \sup_{n\le t\le n+1}|F_t|\longrightarrow0
% \end{equation}
% almost surely. Therefore,
% \begin{equation}
% F_t\longrightarrow0
% \end{equation}
% almost surely as \(t\to\infty\).

The proof is complete.
\end{proof}
\section{Kernel lifting of vortex clusters and quadratic entropy}\label{sec:cluster}
In this section we study how the cluster structure of point vortices appears after the kernel lifting.

Let
\begin{equation}
\mu_t^N
=
\sum_{i=1}^N \kappa_i\delta_{x_i(t)}
\end{equation}
be the empirical vorticity measure, where \(x_i(t)\in\mathbb R^2\) is the position of the \(i\)-th vortex and
\begin{equation}
\kappa_i>0,\qquad \sum_{i=1}^N\kappa_i=1.
\end{equation}
The lifted vorticity is
\begin{equation}
M_t
=
H(\mu_t^N)
=
\int_{\mathbb R^2}h(u)\mu_t^N(du)
=
\sum_{i=1}^N\kappa_i h(x_i(t)).
\end{equation}

The point vortex system\cite{Arnold2021Topological} is
\begin{equation}
\dot x_i(t)
=
\sum_{\substack{j=1\\j\ne i}}^N
\kappa_j
\frac{(x_i(t)-x_j(t))^\perp}
{\|x_i(t)-x_j(t)\|^2},
\qquad
1\le i\le N.
\end{equation}
On every interval without collisions, the trajectories \(x_i(t)\) are smooth.

\begin{definition}
We say that the vortex configuration is \((m,\delta,R)\)-clustered on a time interval \([0,T]\) if there exists a partition
\begin{equation}
\{1,\ldots,N\}
=
\bigcup_{\alpha=1}^m I_\alpha,
\qquad
I_\alpha\cap I_\beta=\varnothing
\quad
\text{for }
\alpha\ne\beta,
\end{equation}
such that, for every \(t\in[0,T]\),
\begin{equation}
\|x_i(t)-x_j(t)\|\le \delta,
\qquad
i,j\in I_\alpha,
\end{equation}
and
\begin{equation}
\operatorname{dist}(C_\alpha(t),C_\beta(t))\ge R,
\qquad
\alpha\ne\beta,
\end{equation}
where $0<\delta\ll 1\ll R$ and
\begin{equation}
C_\alpha(t)=\{x_i(t):i\in I_\alpha\}.
\end{equation}
\end{definition}

For each cluster define its total circulation as
\begin{equation}
K_\alpha
=
\sum_{i\in I_\alpha}\kappa_i.
\end{equation}
If \(K_\alpha\ne0\), define the circulation center by
\begin{equation}
X_\alpha(t)
=
\frac{1}{K_\alpha}
\sum_{i\in I_\alpha}\kappa_i x_i(t).
\end{equation}
The lifted contribution of the \(\alpha\)-th cluster is
\begin{equation}
M_\alpha(t)
=
\sum_{i\in I_\alpha}\kappa_i h(x_i(t)).
\end{equation}

\begin{definition}
For a clustered vortex configuration, the leading lifted cluster mode associated with \(I_\alpha\) is defined by
\begin{equation}
\overline M_\alpha(t)
=
K_\alpha h(X_\alpha(t)).
\end{equation}
The corresponding reduced lifted vorticity is
\begin{equation}
\overline M_t
=
\sum_{\alpha=1}^m
K_\alpha h(X_\alpha(t)).
\end{equation}
\end{definition}

\begin{theorem}
Assume that the vortex configuration is \((m,\delta,R)\)-clustered on \([0,T]\), and assume that \(h\in C^2(\mathbb R^2;H)\) with
\begin{equation}
\|D^2h\|_\infty<\infty.
\end{equation}
Then each lifted cluster satisfies
\begin{equation}
\left\|
M_\alpha(t)-K_\alpha h(X_\alpha(t))
\right\|_H
\le
C_h K_\alpha\delta^2,
\qquad
0\le t\le T,
\end{equation}
where \(C_h>0\) depends only on \(\|D^2h\|_\infty\). Consequently,
\begin{equation}
\left\|
M_t-\sum_{\alpha=1}^m K_\alpha h(X_\alpha(t))
\right\|_H
\le
C_h\delta^2.
\end{equation}
Moreover, for the Gaussian kernel, different lifted cluster modes are almost orthogonal:
\begin{equation}
\left|
(K_\alpha h(X_\alpha(t)),K_\beta h(X_\beta(t)))_H
\right|
\le
K_\alpha K_\beta e^{-R^2/2},
\qquad
\alpha\ne\beta.
\end{equation}
\end{theorem}

\begin{proof}
Since \(h\in C^2(\mathbb R^2;H)\), Taylor expansion at \(X_\alpha(t)\) gives
\begin{equation}
h(x_i(t))
=
h(X_\alpha(t))
+
Dh(X_\alpha(t))(x_i(t)-X_\alpha(t))
+
R_i^\alpha(t),
\end{equation}
where
\begin{equation}
\|R_i^\alpha(t)\|_H
\le
\frac12
\|D^2h\|_\infty
\|x_i(t)-X_\alpha(t)\|^2.
\end{equation}
Multiplying by \(\kappa_i\) and summing over \(i\in I_\alpha\), we get
\begin{equation}
M_\alpha(t)
=
K_\alpha h(X_\alpha(t))
+
Dh(X_\alpha(t))
\sum_{i\in I_\alpha}\kappa_i(x_i(t)-X_\alpha(t))
+
\sum_{i\in I_\alpha}\kappa_iR_i^\alpha(t).
\end{equation}
By the definition of \(X_\alpha(t)\),
\begin{equation}
\sum_{i\in I_\alpha}\kappa_i(x_i(t)-X_\alpha(t))
=
0.
\end{equation}
Hence
\begin{equation}
M_\alpha(t)-K_\alpha h(X_\alpha(t))
=
\sum_{i\in I_\alpha}\kappa_iR_i^\alpha(t).
\end{equation}
Using the cluster diameter assumption,
\begin{equation}
\|x_i(t)-X_\alpha(t)\|\le C\delta,
\end{equation}
we obtain
\begin{equation}
\left\|
M_\alpha(t)-K_\alpha h(X_\alpha(t))
\right\|_H
\le
C_h K_\alpha\delta^2.
\end{equation}
Summing over \(\alpha\) and using
\begin{equation}
\sum_{\alpha=1}^mK_\alpha=1
\end{equation}
gives
\begin{equation}
\left\|
M_t-\sum_{\alpha=1}^mK_\alpha h(X_\alpha(t))
\right\|_H
\le
C_h\delta^2.
\end{equation}
Finally, for \(\alpha\ne\beta\), the separation condition implies
\begin{equation}
\|X_\alpha(t)-X_\beta(t)\|\ge R+O(\delta).
\end{equation}
For the Gaussian kernel,
\begin{equation}
(h(X_\alpha(t)),h(X_\beta(t)))_H
=
\exp\left(-\frac12\|X_\alpha(t)-X_\beta(t)\|^2\right),
\end{equation}
and therefore
\begin{equation}
\left|
(K_\alpha h(X_\alpha(t)),K_\beta h(X_\beta(t)))_H
\right|
\le
K_\alpha K_\beta e^{-R^2/2}.
\end{equation}
This proves the theorem.
\end{proof}

% \begin{remark}
% The theorem shows that the multiscale vortex behavior is visible in \(H\). Vortices inside the same cluster generate almost parallel vectors, while well-separated clusters generate almost orthogonal collective modes. Thus the lifted empirical vorticity has the reduced representation
% \begin{equation}
% M_t
% =
% \sum_{\alpha=1}^m K_\alpha h(X_\alpha(t))
% +
% O(\delta^2).
% \end{equation}
% \end{remark}

\subsection{Quadratic entropy of lifted vortex dynamics}
We now turn to quadratic entropy\cite{Dorogovtsev2010}.  Let
\begin{equation}
A_t^N
=
\{h(x_i(t)):1\le i\le N\}\subset H
\end{equation}
be the lifted vortex configuration.
\begin{definition}
The \emph{quadratic entropy} of a set \( A \subset\mathcal{H} \) is
\begin{equation}
\mathcal{H}_2(A) := \sup \sum_n \|\pi u_n\|^2,
\end{equation}
where the supremum is taken over all finite and infinite sequences of elements of \( A \).
\end{definition}

% Indeed, due to the Gaussian lifting is Lipschitz, if \(N_H(A_t^N,\varepsilon)\) is the covering number in \(H\), then
% \begin{equation}
% N_H(A_t^N,\varepsilon)
% \le
% N_{\mathbb R^2}(\{x_1(t),\ldots,x_N(t)\},\varepsilon).
% \end{equation}
% If the vortices remain in a bounded domain of \(\mathbb R^2\), this gives
% \begin{equation}
% N_H(A_t^N,\varepsilon)
% \le
% C\varepsilon^{-2}.
% \end{equation}
% Hence
% \begin{equation}
% \overline{\dim}_{\mathrm{box}} A_t^N:=
% \limsup_{\varepsilon\downarrow 0}
% \frac{\log N(A_t^N,\varepsilon)}{\log(1/\varepsilon)}\le 2.
% \end{equation}

\begin{theorem}
Assume that the vortex configuration is \((m,\delta,R)\)-clustered on \([0,T]\). For the Gaussian kernel, the lifted vortex configuration satisfies 
\begin{equation}
H_2(A_t^N)
\le
m
+
C N\delta^2
+
C m^2e^{-R^2/2}.
\end{equation}
\end{theorem}

\begin{proof}
For arbitrary \(i\), the length of the orthogonal projection of \(h(x_i(t))\) is bounded by its distance to any other lifted vector. Hence
\begin{equation}
\|\pi_{h(x_i(t))}\|_H^2
\le
\min_{j\ne i}
\|h(x_i(t))-h(x_j(t))\|_H^2.
\end{equation}
Summing over \(i\) and taking the supremum over all finite sequences gives
\begin{equation}
H_2(A_t^N)
\le
\sum_{i=1}^N
\min_{j\ne i}
\|h(x_i(t))-h(x_j(t))\|_H^2.
\end{equation}
Using the Gaussian kernel estimate
\begin{equation}
\|h(x_i(t))-h(x_j(t))\|_H^2
\le
\|x_i(t)-x_j(t)\|^2,
\end{equation}
we obtain
\begin{equation}
H_2(A_t^N)
\le
\sum_{i=1}^N
\min_{j\ne i}
\|x_i(t)-x_j(t)\|^2.
\end{equation}

Now assume the cluster decomposition. Let
\begin{equation}
B_t^m
=
\{h(X_\alpha(t)):1\le\alpha\le m\}.
\end{equation}
The vectors \(h(X_\alpha(t))\) are normalized and almost orthogonal:
\begin{equation}
\|h(X_\alpha(t))\|_H=1,
\end{equation}
and
\begin{equation}
|(h(X_\alpha(t)),h(X_\beta(t)))_H|
\le
e^{-R^2/2},
\qquad
\alpha\ne\beta.
\end{equation}
Therefore
\begin{equation}
H_2(B_t^m)
\le
m+C m^2e^{-R^2/2}.
\end{equation}

For \(i\in I_\alpha\), the Lipschitz property of \(h\) gives
\begin{equation}
\|h(x_i(t))-h(X_\alpha(t))\|_H
\le
\|x_i(t)-X_\alpha(t)\|
\le
C\delta.
\end{equation}
Thus replacing every lifted particle \(h(x_i(t))\) by the corresponding lifted center \(h(X_\alpha(t))\) creates a total squared error bounded by \(C N\delta^2\). Hence
\begin{equation}
H_2(A_t^N)
\le
H_2(B_t^m)
+
C N\delta^2.
\end{equation}
Combining the two estimates yields
\begin{equation}
H_2(A_t^N)
\le
m
+
C N\delta^2
+
C m^2e^{-R^2/2}.
\end{equation}
This completes the proof.
\end{proof}

\bibliographystyle{unsrt}
\bibliography{ref}
\end{document}